\documentclass[]{article}
\usepackage{amsmath,amsthm,nicefrac,xcolor,citesort}
\usepackage{bm,amssymb,mathrsfs,url,units,dsfont}
\usepackage{setspace}%
\DeclareMathOperator\Cov{Cov}
\DeclareMathOperator\Corr{Corr}
\DeclareMathOperator\Var{Var}
\DeclareMathOperator\Dimh{Dim_{_{\rm H}}}
\DeclareMathOperator\Dimm{Dim_{_{\rm M}}}
\DeclareMathOperator\lDimh{\underline{\rm Dim}_{_{\rm H}}}
\DeclareMathOperator\Den{Den_\nu}
\newcommand{\1}{\mathds{1}}
\renewcommand{\P}{\mathds{P}}
\newcommand{\E}{\mathds{E}}
\newcommand{\R}{\mathds{R}}
\newcommand{\Z}{\mathds{Z}}
\newcommand{\N}{\mathds{N}}

\newcommand{\cN}{\mathcal{N}}
\newcommand{\cS}{\mathcal{S}}
\newcommand{\cV}{\mathcal{V}}
\newcommand{\cP}{\mathcal{P}}

\renewcommand{\d}{\mathrm{d}}
\newcommand{\e}{\mathrm{e}}

\newcommand{\dimh}{\dim_{_{\rm H}}}

\newtheorem{obs}{Observation}
\newtheorem{stat}{Statement}[section]
\newtheorem{proposition}[stat]{Proposition}
\newtheorem{corollary}[stat]{Corollary}
\newtheorem{theorem}[stat]{Theorem}
\newtheorem{lemma}[stat]{Lemma}
\theoremstyle{definition}
\newtheorem{definition}[stat]{Definition}
\newtheorem{remark}[stat]{Remark}

\numberwithin{equation}{section}

\renewcommand{\geq}{\geqslant}
\renewcommand{\leq}{\leqslant}
\renewcommand{\ge}{\geqslant}
\renewcommand{\le}{\leqslant}

\begin{document}

\title{Intermittency and multifractality:\\
	A case study via parabolic stochastic PDEs%
	\thanks{Research supported in part by NSF grant DMS-1307470.}
}
\author{Davar Khoshnevisan\\University of Utah
	\and
	Kunwoo Kim\\University of Utah
	\and
	Yimin Xiao\\Michigan State University
}

\date{Last update: March 20, 2015}
\maketitle

\begin{abstract}
	Let $\xi$ denote space-time white noise,
	and consider the following stochastic partial differential equations:
	(i) $\dot{u}=\frac12 u'' + u\xi$, started identically at one; 
	and (ii) $\dot{Z}=\frac12 Z'' + \xi$, started identically at zero.
	It is well known that the solution to (i) is intermittent, whereas the solution
	to (ii) is not. And the two equations are known to be in different
	universality classes. 
	
	We prove that the tall peaks of both 
	systems are multifractals in a natural large-scale sense. Some of this work is 
	extended to also establish the multifractal behavior of the peaks of
	stochastic PDEs on $\R_+\times\R^d$ with $d\ge 2$. G. Lawler
	has asked us if intermittency is the same as multifractality.
	The present work gives a negative answer to this question. 
	
	As a byproduct of  our methods, we prove also that the peaks of the Brownian motion
	form a large-scale monofractal, whereas the peaks of the
	Ornstein--Uhlenbeck process on $\R$ are multifractal.
	
	Throughout,
	we make extensive use of the macroscopic fractal theory of M.T. Barlow and 
	S.J. Taylor \cite{BarlowTaylor,BarlowTaylor1}. We expand on aspects of 
	the Barlow--Taylor theory, as well.\\
	
	
\end{abstract}

\section{Introduction, and main results}

The principle aim of this article is to answer the following question 
that was posed to us by Gregory Lawler in January of 2012 [private communication]: 
 {\it Is  ``intermittency'' the same property as
``multifractality''}? We will argue below that the short answer is no.
Among other things, it follows that the macroscopic analysis of disordered systems
can unravel a great deal more complexity than its microscopic counterparts
(as compared with the theory of Paladin et al \cite{PaladinPelitiVulpiani}, for example).

The two quoted terms, ``intermittency'' and ``multifractality,'' 
appear also in the title of this article. They
are meant to be understood as informal descriptions of behavior that is commonly
observed in a vast array of complex scientific problems in which there
are infinitely-many natural length scales. This sort of behavior is common in,
but not limited to, problems in full-blown turbulence.

Intermittency is a well-defined mathematical property
which we recall first. Multifractality will be treated  afterward, and will
require more effort. In fact, one of the novel parts of this paper
is to set up a mathematical framework within which we can understand
macroscopic multifractality in a way that is meaningful in the present context. 

In order to motivate some of the results of this paper
let us consider the eigenvalue problem for the random heat operator,
\begin{equation}
	{\rm H} := \frac{\partial}{\partial t} - \frac12\frac{\partial^2}{\partial x^2}
	- {\rm M}\qquad(t>0,x\in\R),
\end{equation}
acting on space-time functions $f=f(t\,,x)$ such that $f(0\,,x)\equiv1$
for all $x\in\R$, say.
Here, ${\rm M}$ denotes the random multiplication operator,
defined via
\begin{equation}
	({\rm M} f)(t\,,x) := f(t\,,x)\xi(t\,,x)
	\qquad(t>0,x\in\R),
\end{equation}
where $\xi$ is a space-time white noise. Because $\xi$ is a random
Schwartz-type distribution, and not a nice classically-defined random process, 
the operator ${\rm M}$ needs to be understood
in integrated form: As a Wiener-integral map when $f$ is nonrandom; and 
more generally as 
a Walsh-type stochastic integral operator when $f$ is a predictable random field.

It is well known, and also easy to see, that
the spectrum of $\rm{H}$ is all of $\R$, and that the eigenfunction 
of $\rm H$ that corresponds
to eigenvalue $\lambda\in\R$ is $\exp(-\lambda t) u(t\,,x)$, where 
$u=u(t\,,x)$ solves the \emph{parabolic Anderson model},
which we may understand rigorously
as the solution to the following Walsh-type
stochastic partial differential equation \cite{Walsh}:
$$\left[\begin{split}
	&\dot{u}= \tfrac12u'' + u\xi\quad\text{on $(0\,,\infty)\times\R$,}\\
	&\text{subject to }u(0\,,x)=1\text{ for all $x\in\R$}.
\end{split}\right.
\eqno({\rm PAM})
$$

It is well known that the parabolic Anderson model (PAM) 
has a unique solution which is a predictable random field that is continuous
in both variables $t$ and $x$.\footnote{The term predictable is used as in
Walsh \cite{Walsh}, and refers to predictability with respect to the filtration
generated by the white noise $\xi$.} Moreover, there exist finite
and positive constants $U_0,U,L_0,L$ such that 
\begin{equation}\label{moments:intro}
	L_0^k \exp\left( Lk^3t\right) \le \E\left(|u(t\,,x)|^k\right)\le
	U_0^k \exp\left( Uk^3t\right),
\end{equation}
valid uniformly for all real numbers $t>0$, $k\ge 2$, and $x\in\R$.
[In fact, these moments do not depend on the value of $x$.] 
See for example Chapter 5 of \cite{cbms} for a self-contained account.
As a consequence, the moment Lyapunov exponents of the solution $u$
are strictly positive
and finite, where the \emph{$k$th moment Lyapunov exponent} of
$u$ is defined as
\begin{equation}
	\lambda(k) := \lim_{t\to\infty} t^{-1} \log\E\left(|u(t\,,0)|^k\right),
\end{equation}
for every real number $k>0$.

One can devise a subadditivity argument in order to show 
that these Lyapunov exponents exist. It is widely believed that the following 
\emph{Kardar formula}
holds \cite{Kardar,KPZ,KZ}:\footnote{%
	Khoshnevisan \cite[Chapter 6]{cbms} has shown that 
	$L\ge k(k^2-1)/24$ for every integer $k\ge 2$. Thus, in particular,
	$\lambda(k)\ge k(k^2-1)/24$ for all integers $k\ge2$.
	Borodin and Corwin \cite{BorCor} have found the first completely-rigorous
	proof of \eqref{Kardar} in the closely-related case that $k\ge 2$
	is an integer and the initial value of (PAM) is replaced by $\delta_0$.}
\begin{equation}\label{Kardar}
	\lambda(k)= k(k^2-1)/24\quad\text{for all real numbers
	$k>0$}.
\end{equation}
The \emph{a priori} bounds \eqref{moments:intro} and convexity
considerations together imply that
\begin{equation}\label{intermittency}
	k\mapsto k^{-1}\lambda(k)
	\quad\text{is strictly increasing on $[2\,,\infty)$.}
\end{equation}
This property is known as \emph{intermittency} and 
was referred to earlier on in a broader context. 

It is possible to argue that because of intermittency the solution to the parabolic
Anderson model (PAM) has very tall peaks on infinitely-many different length scales.
We have not yet described very precisely what it means to ``have tall peaks
on infinitely-many different scales.'' Still, the adjective ``multifractal''
is supposed to reflect the presence of such behavior.
The interested reader can find different heuristic accounts of
how intermittency might explain multifractality in
the Introductions of Bertini and Cancrini \cite{BertiniCancrini}
and Carmona and Molchanov \cite{CM}.
Chapter 7 of Khoshnevisan \cite[\S7.1]{cbms} yields a related
but slightly more precise explanation. A recent paper by
Gibbon and Titi \cite{GT}
contains an account of how, and why, intermittency and other attempts
at describing ``multifractality'' arise naturally in a large number
of multi-scale problems in science.

Now that we recalled ``intermittency,'' within context,
we begin to propose a mathematical model of ``multifractality'' 
that is general enough for our later needs, but concrete enough
that it is amenable to exact analysis. It turns out to be helpful to
not begin with a too-concrete setting such as stochastic PDEs.
Therefore, let $X:=\{X(t)\}_{t\in\R^n_+}$ be a  
real-valued stochastic process with continuous trajectories.
We are interested in saying that $X$
has tall peaks that are ``multifractal'' in a large-scale sense.
In order to do this, let us posit that the tall peaks of the stochastic process $X$ are
described by a \emph{gauge function} $g:\R_+\to\R_+$. By this we
mean that $g$ is a nonrandom and increasing function such that
$\lim_{r\to\infty} g(r)=\infty$ and
\begin{equation}\label{LIL:X}
	\limsup_{\|t\|\to\infty} \frac{|X(t)|}{g(\|t\|)}=1
	\qquad\text{a.s.}
\end{equation}
That is, we assume that there exists a \emph{non-random} 
gauge function $g$ which describes
the largest possible scale on which we can compute the tall
peaks of $X$ in a macroscopic sense. Let us fix the gauge function $g$
in our mind, and then consider the random set,
\begin{equation}\label{L_X}
	\cP_X(\gamma) := \left\{ t\in\R^n_+:\ \|t\|>c_0,\
	\frac{|X(t)|}{g(\|t\|)} > \gamma \right\},
\end{equation}
where $\gamma>0$ is a tuning parameter and $c_0$
is a fixed positive constant. We may think of
every point $t\in\cP_X(\gamma)$ as a \emph{tall peak}
of the process $X$---suitably normalized by $g$--- \emph{viewed
in length scale $\gamma$}. 

Because $\cP_X(\gamma_1)\subseteq\cP_X(\gamma_2)$
when $\gamma_1>\gamma_2$, we might expect the
existence of a phase transition that does not depend on our prior choice of
$c_0$. This is correct, and not difficult to explain.
In fact, the growth condition \eqref{LIL:X}
implies  that there exists a uniquely-defined transition point at
$\gamma=1$:
$\cP_X(\gamma)$ is a.s.\ unbounded when $\gamma<1$
whereas $\cP_X(\gamma)$ is a.s.\ bounded when $\gamma>1$.
Motivated by this simple observation, let us introduce the following.

\begin{definition}\label{def:multifractal}
	Choose and fix a gauge function $g$ and
	consider the tall peaks of $X$, as defined in \eqref{L_X}.
	We say that the tall peaks of $X$ are 
	\emph{multifractal} when there exist 
	infinitely-many length scales $\gamma_1>\gamma_2>\cdots>0$
	such that with probability one,
	\begin{equation}\label{DimDim:H}
		\Dimh\left(\cP_X(\gamma_{i+1})\right) 
		< \Dimh \left( \cP_X(\gamma_i) \right)
		\quad\text{for all $i\ge 1$}.
	\end{equation}
	If, on the other hand,
	\begin{equation}
		\Dimh(\cP_X(\gamma)) = \begin{cases}
			1&\text{whenever $\gamma<1$},\\
			0&\text{whenever $\gamma>1$},
		\end{cases}
	\end{equation}
	then we say that the tall peaks of $X$ are \emph{monofractal}.\footnote{In
	principle, it can happen that the tall peaks of $X$ are neither multifractal
	nor monofractal. That is, when \eqref{DimDim:H} holds for finitely-many
	$\gamma_1>\gamma_2>\cdots>\gamma_n$. It might be interesting to
	construct non-trivial examples of stochastic processes whose tallest peaks
	are of this latter type. We aren't aware of any natural examples at this time.
	}
\end{definition}

In the preceding, $\Dimh (E)$ denotes the macroscopic Hausdorff dimension of
$E\subset\R^n$. This notion is due to
Barlow and Taylor \cite{BarlowTaylor,BarlowTaylor1}, and  will be reviewed
in the next section. For the time being, it suffices to know only that
$\Dimh(E)$ is a real number between $0$ and $n$,
and describes the large-scale geometry of $E\subset\R^n_+$
in a way that parallels how the usual notions of fractal dimension
try to describe small-scale geometry.

In order to see how our definition of multifractality can be used in 
stochastic PDEs, let us return to the stochastic heat equation, (PAM). A deep theorem of
Mueller \cite{Mueller} asserts that 
\begin{equation}\label{eq:pos:PAM}
	u(t\,,x) >0\quad\text{for all $t\ge0$ and $x\in\R$},
\end{equation}
outside a single event of 
probability zero. Because the natural logarithm is strictly monotonic, 
the tall peaks of the random function $x\mapsto u(t\,,x)$ are in one-to-one
correspondence with the tall peaks of $x\mapsto X_t(x)$
at all times $t>0$, where
\begin{equation}\label{X:pam}
	X_t(x) := \left( \frac{32}{9t}\right)^{1/3}\log u(t\,,x).
\end{equation}
Conus et al \cite{CJK} have proved that for all $t>0$, 
\begin{equation}\label{CJK:pam}
	0<\limsup_{x\to\infty} \frac{X_t(x)}{g(x)}<\infty\quad
	\text{a.s., where } g(x) := (\log_+ x)^{2/3}.
\end{equation}
More recently,  Chen \cite{Chen} has found the following 
improvement to \eqref{CJK:pam}: For all $t>0$,
\begin{equation}\label{Chen:pam}
	\limsup_{x\to\infty} \frac{X_t(x)}{g(x)} =1
	\qquad\text{a.s.}
\end{equation}
In other words,
$g(x) := (\log_+ x)^{2/3}$  is a natural gauge function
for measuring the tall peaks of the random height function 
$x\mapsto X_t(x)$. 
The particular normalization of $X$ in \eqref{X:pam} is there to merely
ensure that the lim sup is one, which matches the form of \eqref{Chen:pam}
with that of \eqref{LIL:X}. Having said this, we are ready to mention one of
the results of this work.

\begin{theorem}\label{th:multifractal:pam}
	Consider the gauge function $ g(x) := (\log_+ x)^{2/3}$
	and, for all $t>0$, define $X_t = \{X_t(x), \, x \in \R\}$ as in \eqref{X:pam}.
	Then the tall peaks of $X_t$ are multifractal with probability one
	for all $t>0$. In fact, for every $t,\gamma>0$,
	\begin{equation}\label{multifractal:spectrum:pam}
		\Dimh \left( \cP_{X_t}(\gamma)\right) =
		1 - \gamma^{3/2}\quad\text{a.s.,}
	\end{equation}
	where $\Dimh(E)<0$ means that $E$ is bounded.
\end{theorem}

Theorem \ref{th:multifractal:pam} shows that the tall peaks of the
solution $u$ to (PAM) are multifractal in the sense
described more precisely in Theorem \ref{th:multifractal:pam}. 
And we can think of \eqref{multifractal:spectrum:pam}
as a description of large-scale ``multifractal spectrum'' of the peaks of 
the random field $u$. Theorem \ref{th:multifractal:pam} is a consequence
of a much deeper result which says, using informal language, that the tall peaks of
the solution to every known multifractal 
nonlinear stochastic PDE \cite{FA:whitenoise} are generically multifractal. 
See Theorem \ref{co:PAM:WN}
and Corollary \ref{co:LIL:pam} below.

Earlier we mentioned  that $u$ is
intermittent in the sense of \eqref{intermittency}.
Together with Theorem \ref{th:multifractal:pam},
this assertion says that intermittency and multifractality
can coexist. Next we describe a \emph{non-intermittent} system whose
tall peaks are also multifractal. Consequently, non-intermittency
and multifractality can also co-exist, whence follows  the negative answer
to Professor Lawler's question with which we began.

Consider the linear stochastic heat equation,
$$\left[\begin{split}
	&\dot{Z}= \tfrac12Z'' + \xi\quad\text{on $(0\,,\infty)\times\R$,}\\
	&\text{subject to }Z(0\,,x)=0\text{ for all $x\in\R$}.
\end{split}\right.
\eqno({\rm HE})
$$
It is well known that a solution exists, is unique,
and is a continuous centered Gaussian random field. 
It is also possible to prove that the moment Lyapunov exponents
of $Z$ are all zero, and so $Z$ is not an intermittent process.\footnote{This
can be deduced from inspecting the proof of Theorem 2.1 of
ref.\ \cite{FA:whitenoise},
for example.}

Also, one can prove fairly easily that for all $t>0$
the following holds with probability one:
\begin{equation}
	\limsup_{x\to\infty} \frac{Z(t\,,x)}{g(x)} = 
	(t/\pi)^{1/4}\quad\text{a.s.,
	where $g(x):=(2\log_+ x)^{1/2}$.}
\end{equation}
A ``steady state'' version of this fact
appears in print, for example, in  
Collela and Lanford \cite[Theorem 1.1(c)]{CollelaLanford}.
And the fact itself follows by specializing an even earlier, very general, theory
of Pickands \cite{Pickands}; see also Qualls and Watanabe \cite{QW}.

Let
\begin{equation}\label{X:he}
	X_t(x) := (\pi/t)^{1/4} Z(t\,,x).
\end{equation}
The following shows that the tall peaks of $X_t$---hence also
those of $Z(t\,,\cdot)$---are multifractal even though $Z$
is non-intermittent.

\begin{theorem}\label{th:multifractal:he}
	Consider the gauge function $ g(x) := (\log_+ x)^{1/2}$
	and, for all $t>0$, define $X_t = \{X_t(x), \, x \in \R\}$  as in \eqref{X:he}.
	Then the tall peaks of $X_t$ are multifractal with probability one
	for all $t>0$. In fact, for every $t,\gamma>0$,
	\begin{equation}
		\Dimh \left( \cP_{X_t}(\gamma)\right) =
		1-\gamma^2\quad\text{a.s.,}
	\end{equation}
	where $\Dimh(E)<0$ means that $E$ is bounded.
\end{theorem}


Theorems \ref{th:multifractal:pam} and \ref{th:multifractal:he}
are both consequences of two more general theorems
about multifractal random fields [Theorems \ref{th:Gen:LIL:UB}
and \ref{th:Gen:LIL:LB} below]. Those general theorems have
other interesting consequences as well. Let us mention one such
result.

\begin{theorem}\label{B:OU}
	Let $B$ denote a one-dimensional Brownian motion and
	$U(t)$ $:= \exp(-t/2)B(\e^t)$ an Ornstein--Uhlenbeck process
	on $\R$. Then, the tall peaks of $B$ are a monofractal,
	whereas those of $U$ are multifractal in the following sense:
	For every $\gamma>0$, 
	\begin{equation}\label{B:B}
		\Dimh\left\{ s\ge \e^\e:\ \frac{B(s)}{(2 s\log\log s)^{1/2}}
		\ge \gamma\right\} = \begin{cases}
			1&\text{if $\gamma\le1$},\\
			0&\text{if $\gamma>1$},
		\end{cases}
	\end{equation}
	almost surely, whereas
	\begin{equation}\label{OU:OU}
		\Dimh\left\{ s\ge \e^\e:\ \frac{U(s)}{(2 \log s)^{1/2}}
		\ge \gamma\right\} =1-\gamma^2,
	\end{equation}
	almost surely, where we recall $\Dimh(E)<0$ means
	that $E$ is bounded.
\end{theorem}
Theorem \ref{B:OU} will be proved in two parts: \eqref{B:B} is proved
below in Theorem \ref{th:LIL:BM}; and \eqref{OU:OU} is proved in
Theorem \ref{th:U}.

We have included a final \S\ref{sec:pamcolor} wherein we state and prove
a suitable variation of  Theorem \ref{th:multifractal:pam}, in which the stochastic
partial differential equation (PAM) is replaced by a similar-looking object with
$u''$ replaced by $\Delta u$, the Laplacian, the latter acting on
a space variable $x\in\R^d$ in place of $\R$. And space-time white noise is
replaced by a centered Gaussian noise that is white in time and suitably
correlated in space  to ensure the existence and uniqueness of a well-tempered
solution. Among other things,
such equations are well-known
models that play a role in the large-scale structure of the universe;
see \S\ref{sec:pamcolor} for more details.

Let us conclude the Introduction with a word on notation. From now on
we will write $\phi_t(x)$ and $\phi(t\,,x)$ interchangeably, depending
on which is more convenient, for any space-time function [or even generalized
function/noise] $\phi$. This  choice ought
not inconvenience the reader, since $\phi_t$ is {\it never} used to denote
the time derivative of $\phi$; rather, it is the standard probabilistic notation
for describing the time evolution of $\phi$.

\section{Dimension and density}
Let us begin by recalling the Barlow--Taylor definition of the
macroscopic dimension \cite{BarlowTaylor,BarlowTaylor1}
$\Dimh E$ of a set $E\subseteq\R^d$.

Define, for all integers $n\ge 0$,
\begin{equation}
	\cV_n :=\left[-\e^n\,, \e^n\right)^d, \quad
	\cS_0 :=\cV_0,  \quad\text{and}\quad
	\cS_{n+1}  :=\cV_{n+1}\setminus\cV_n.
\end{equation}
We might sometimes refer to $\cS_n$ as the \emph{$n$th shell} in $\R^d$. 

An important idea of Barlow--Taylor \cite{BarlowTaylor,BarlowTaylor1}---see
also Naudts \cite{Naudts} for a precursor to this idea---%
is to construct a family of Hausdorff-type contents on each shell, 
and then use the totality of those contents in order to define a family of Hausdorff-type
contents on all of $\R^d$. Once this is done, a notion of macroscopic Hausdorff dimension
presents itself quite naturally. 

\begin{definition}\label{def:B}
	Let $\mathcal{B}$ denote the collection of all sets of the form
	\begin{equation}\label{box}
		Q(x\,,r) := \left[ x_1\,,x_1+r\right)\times\cdots\times
		\left[ x_d\,,x_d+r\right),
	\end{equation}
	as $x :=(x_1\,,\ldots,x_d)$ ranges in $\R^d$ and $r$  in 
	$(0\,,\infty)$. If $Q:=Q(x\,,r)$ is an element of $\mathcal{B}$,
	then we may refer to $Q$ as an \emph{upright box} with \emph{southwest corner}
	$x $ and \emph{sidelength} ${\rm side}(Q):=r$.
\end{definition}

Choose and fix some number $c_0>0$,
and define for  every set $E\subseteq \R^d$,
all real numbers $\rho>0$, and each integer $n\ge 0$,
\begin{equation}\label{nu:n:rho}
	\nu^n_\rho(E) := \inf  \sum_{i=1}^m\left(\frac{{\rm side}(Q_i)}{\e^n}\right)^\rho,
\end{equation}
where the infimum is taken over all upright boxes $Q_1,\ldots,Q_m$
of side $\ge c_0$ that cover $E\cap\cS_n$. 
We may think of $\nu^n_\rho(E)$ as the restriction to the $n$th shell $\cS_n$ of the
scaled $\rho$-dimensional \emph{Hausdorff content} of $E$.

\begin{definition}
	The Barlow--Taylor \emph{macroscopic Hausdorff dimension} 
	of $E\subseteq\R^d$ is 
	\begin{equation} \label{Def:Hdim}
		\Dimh E :=  \Dimh (E):=
		\inf\left\{\rho>0:\ \sum_{n=1}^\infty\nu^n_\rho(E)
		<\infty\right\}.
	\end{equation}
\end{definition}

By (\ref{Def:Hdim}), any bounded set $E \subseteq \R^d$ has $\Dimh E = 0.$
We will leave the following simple fact as exercise for the interested
reader.

\begin{lemma}
	The numerical value of $\Dimh E$ does not depend on $c_0>0$.
\end{lemma}

Thus, we could choose $c_0=2$ or $c_0=\e/\sqrt 2$, in our definition
of $\nu^n_\rho$, without affecting the value of $\Dimh E$. It
is important to point out though that $c_0=0$ can lead to a different
value of $\Dimh E$. 

In order to see why
we are ruling out the possibility of $c_0=0$,
let us define for all $E\subseteq\R^d$ and $r>0$
a new set $E^z(r)\subseteq\Z^d$
as follows:
\begin{equation}
	E^z(r) := \left\{ x\in\Z^d:\, E\cap Q(x \,,r)\neq\varnothing\right\}.
\end{equation}
The notation is basically due to Barlow and Taylor,
who observed the following \cite[Lemma 6.1]{BarlowTaylor1}
but stated it using slightly different language: For all $E\subseteq\R^d$
	and $r>0$,
\begin{equation} \label{Eq:BT92}
	\Dimh E = \Dimh(E^z(r)),
\end{equation}
where $\Dimh$ on the right-hand side of  (\ref{Eq:BT92}) is the discrete 
Hausdorff dimension on $\Z^d$ of Barlow and Taylor \cite{BarlowTaylor,BarlowTaylor1}.
In other words, because we chose only covers of $E$ that have
side $\ge c_0>0$, the local structure of $E$ does not affect the value of
its macroscopic Hausdorff dimension. Put yet in another way, this shows
that the Barlow--Taylor definition of
\emph{$\Dimh E$ quantifies the large-scale geometry of $E$}, without
obstruction by the microscopic structure of the set $E$.

Since the particular value of $c_0>0$ does not matter, from now on 
we follow the choice of Barlow and Taylor \cite{BarlowTaylor,BarlowTaylor1},
and set
\begin{equation}
	c_0=1.
\end{equation}

The following is a macroscopic counterpart of a familiar
result about microscopic Hausdorff dimension.

\begin{lemma}\label{lem:Lipschitz}
	Suppose $f:\ E\to\R^p$ is 	a Lipschitz function on  $E\subseteq\R^d$ and satisfies the 
growth condition $\liminf\limits_{x \in E, \, |x|\to \infty} |f(x)|/|x| >0$.  Then,
	\begin{equation}\label{Eq:LipUpp}
		\Dimh f(E) \le \Dimh E.
	\end{equation}
In particular, if $f:\R^d\to\R^d$ is bi-Lipshitz on $E \subseteq \R^d$; that is, there exists a  
positive constant $L \ge 1$ such that 
\[
L^{-1} |x-y| \le |f(x) - f(y) | \le L |x-y|\ \ \hbox{ for all }\, x, y \in E.
\]
Then
\begin{equation} \label{Eq:LipEq}
		\Dimh f(E)=\Dimh E. 
	\end{equation}
\end{lemma}

\begin{proof}
	The proof is similar to that of the same assertion for
	ordinary Hausdorff dimension:
	For every $\rho>\Dimh E$, we can find upright boxes
	$Q_{j,n}\in\mathcal{B}$---%
	indexed by $1\le j\le m_n$, $n\ge 1$---%
	with ${\rm side}(Q_{j,n})\ge1$
	such that: 
	\begin{itemize}
		\item[(i)] $E\cap\cS_n \subseteq\cup_{j=1}^{m_n}Q_{j,n}$
			for all $n\ge 1$; and 
		\item[(ii)] $\sum_{n=1}^\infty
			\sum_{j=1}^{m_n}({\rm side}(Q_{j,n})/\e^n)^\rho<\infty$.
	\end{itemize}
	
	It is clear that  $f(E)\subset \cup_{n=1}^\infty \cup_{j=1}^{m_n} f(Q_{j,n})$. 
	Since $f$ is Lipschitz on $E$ we can find a finite and positive constant $q$
	such that
	\begin{equation}
		|f(x)-f(y)|\le q|x-y|
		\qquad\text{for all $x,y\in E$.}
	\end{equation}
	Without loss of generality, we may assume that $q\ge1$ is an integer;
	otherwise, we replace $q$ by $1+\lfloor q\rfloor$ everywhere.
	In particular, it follows that
	every $f(Q_{j,n})$ can be covered with an upright box
	whose sidelength is between $1$
	and $q\,{\rm side}(Q_{j,n}).$
	In addition, the growth condition on $f$ implies that there is a constant $\varepsilon > 0$ such that 
	$f(x)\ge \varepsilon \e^n$ for every $x \in E\cap \cS_n$
	This implies that
	\begin{equation}
		\sum_{n=0}^\infty \nu^n_\rho(f(E)) \le \sum_{n=0}^\infty q^\rho
		\sum_{j=1}^{m_n}\left( \frac{{\rm side}(Q_{j,n})}{\varepsilon  \e^n}
		\right)^\rho\qquad\text{for all $n\ge 1$},
	\end{equation}
	whence $\sum_{n=0}^\infty \nu^n_\rho(f(E))<\infty$ by (ii).
	This proves that $\Dimh f(E)\le \rho$ for all $\rho>\Dimh E$ and
	implies \eqref{Eq:LipUpp}.
	
	Finally, the bi-Lipshitz condition implies that both $f$ and its inverse $f^{-1}$ on $f(E)$ satisfy 
	the growth condition.  Hence, we make two appeals to the first part of Lemma \ref{lem:Lipschitz},
	once for $f$ and once for $f^{-1}$, to see that  \eqref{Eq:LipEq} holds.
\end{proof}


Lemma \ref{lem:Lipschitz} is the large-scale/macroscopic analogue of the following 
well-known fact: If $f:E\to\R$ is locally Lipschitz continuous, then
\begin{equation}
	\dimh f(E)\le\dimh E,
\end{equation} 
where ``$\dimh$'' denotes the usual [microscopic] Hausdorff dimension in $\R^d$.
Let us, however, observe that  \eqref{Eq:LipUpp} does not hold when
$f$ is only Lipschitz. For example, set $f(x):=\ln (x)$ for $x \ge 1$ and $E = \exp(\N)$ to see that
\begin{equation}
	\Dimh f(E) = \Dimh \N=1 > 0 = \Dimh( \exp(\N)).
\end{equation}
In the above, $\N:=\{1\,,2\,,\ldots\}$ denotes as usual the set of all natural
numbers. We will present an interesting example of $E$ in Remark \ref{rem:compare:B:U} below
which shows that \eqref{Eq:LipUpp} does hold for $f(x):=\ln (x)$ even though $f$ does not satisfies 
the growth condition in Lemma  \ref{lem:Lipschitz}.

Next, let us mention a technical estimate, which is
a ``density theorem.'' The following is
a large-scale analogue of the classical Frostman lemma,
and basically rephrases Theorem
4.1(a) of Barlow and Taylor \cite{BarlowTaylor1}
in a slightly different form that is more convenient for us.

\begin{lemma}[A Frostman-type lemma]\label{lem:Frostman}
	Choose and fix an integer $n\ge 1$,
	and suppose $E\subset\cS_n$ is a Borel set in $\R^d$. Let
	$\mu$ denote a finite non-zero Borel measure on $E$, and define for all $\rho\ge 0$,
	\begin{equation}\label{eq:K:rho}
		K_\rho := \sup\left\{ \frac{\mu(Q)}{[{\rm side}(Q)]^\rho}:\,
		Q\in\mathcal{B},\ Q\subset\cS_n,\ {\rm side}(Q)\ge 1\right\}.
	\end{equation}
	Then, $\nu_\rho^n(E) \ge  K_\rho^{-1}\e^{-n\rho}\mu(E)$.
\end{lemma}

\begin{remark}
	The constant $K_\rho$ of Lemma \ref{lem:Frostman} typically depends on
	$n$ as well, and is always finite and positive. 
\end{remark}

Let $\nu$ be a Borel measure on $\R^d$.
For all $E\subseteq\R^d$ let $\Den E$ denotes the \emph{upper density}
of $E$ \emph{with respect to $\nu$}. That is,
\begin{equation}
	\Den E := \Den(E) := \limsup_{t\to\infty} \frac{%
	\nu\left( E\cap [-t\,,t]^d\right)}{(2t)^d}.
\end{equation}

The following describes
an easy-to-verify sufficient condition for a set $E$ in $\R^d$ to have full macroscopic 
Hausdorff dimension. 

\begin{lemma}\label{lem:Den:Dim}
	Let $\nu$ denote either the Lebesgue measure on $\R^d$
	or counting measure on a sublattice $\varepsilon\Z^d$
	of $\R^d$ for some $\varepsilon>0$. Then, for all Borel sets
	$E\subseteq\R^d$, if \ $\Den E>0$, then $\Dimh E=d$.
\end{lemma}

\begin{proof}
	We will prove the lemma only in the case that $\nu$ 
	denotes the Lebesgue measure. When
	$\nu$ is the counting measure on $\varepsilon\Z^d$, the
	result is proved in almost exactly the same way. 
	
	Barlow and Taylor \cite{BarlowTaylor,BarlowTaylor1} have introduced
	another large-scale notion of Hausdorff dimension, which we write 
	as follows:
	\begin{equation}\label{lDimh}
		\lDimh E :=  \lDimh (E):=
		\inf\left\{\rho>0:\ \lim_{n\to\infty}\nu^n_\rho(E)=0\right\}.
	\end{equation}
	It is easy to verify that $\lDimh E\le \Dimh E$ for all $E\subseteq\R^d$; therefore we might think
	of $\lDimh$ as the \emph{lower Hausdorff dimension}, in the macroscopic sense.
	Our goal is to prove the following somewhat stronger statement:
	\begin{equation}\label{eq:DenDim}
		\Den E>0\quad\Rightarrow\quad \lDimh E\ge d.
	\end{equation}
	It follows from \eqref{eq:DenDim} that
	$\Dimh E\ge d$.
	Since $\Dimh E\le \Dimh(\R^d)=d$ (see Barlow and Taylor
	\cite[Example 4.1]{BarlowTaylor}), \eqref{eq:DenDim}
	completes the proof.
	
	For every real number $a>1$ and integers $n\ge 0$ define
	\begin{equation}\begin{split}	
		\cV_n(a)&:=\left[-a^n\,, a^n\right)^d, \\
		\ \cS_0(a)&:=\cV_0(a),  \\
		\cS_{n+1}(a) & :=\cV_{n+1}(a)\setminus\cV_n(a).
	\end{split}\end{equation}
	Note, in particular, that $\cS_n=\cS_n(\e)$
	and $\cV_n=\cV_n(\e)$. Define
	for  every set $E\subset\R^d$,
	all real numbers $\rho>0$, and each integer $n\ge 0$,
	\begin{equation}
		\nu^n_\rho(E;a) := \min  \sum_{i=1}^m
		\left(\frac{{\rm side}(Q_i)}{a^n}\right)^\rho,
	\end{equation}
	where the minimum is taken over all upright boxes $Q_1,\ldots,Q_m$
	of side $\ge 1$ that cover $E\cap\cS_n(a)$. 
	
	Barlow and Taylor \cite[p.\ 127]{BarlowTaylor}
	have remarked that their
	macroscopic and lower Hausdorff dimensions do not depend on the
	choice of $a$; in particular,
	\begin{equation}\label{Dim:a}
		\lDimh E = 
		\inf\left\{\rho>0:\ \lim_{n\to\infty}\nu^n_\rho(E;a)
		=0\right\}\qquad\text{for all $a>1$}.
	\end{equation}
	In fact, the original
	construction of Barlow and Taylor  \cite{BarlowTaylor,BarlowTaylor1} 
	is similar to ours, but with $a=2$ and not $a=\e$, as is the case here.
	
	We may define a Borel measure $\mu$ on $E$ by setting
	\begin{equation}
		\mu(G) := |E\cap G|\qquad\text{for all $G$,}
	\end{equation}
	where $|\cdots|$ denotes the Lebesgue measure.
	Clearly,
	\begin{equation}
		K_d :=\sup\left\{ \frac{\mu(Q)}{r^d}:
		Q\in\mathcal{B},\ Q\subseteq\cS_n,\ {\rm side}(Q)\in[1\,,r],\ r\ge 1\right\}
		\le 1.
	\end{equation}
	Furthermore, 
	\begin{equation}
		\mu(\cV_n(a))\le |\cV_n(a)|=2^d a^{nd}
		\qquad\text{for all $n\ge 0$,} 
	\end{equation}
	and
	\begin{equation}
		\mu(\cV_{n+1}(a))\ge 2^d(\Den E - \delta_n)a^{(n+1)d},
	\end{equation} 
	for infinitely-many integers $n\ge 1$,
	where $\{\delta_n\}_{n=1}^\infty$ is a sequence that
	satisfies $\lim_{n\to\infty} \delta_n= 0$. Therefore,
	\begin{equation}
		\mu(\cS_{n+1}(a)) = \mu(\cV_{n+1}(a)) - \mu(\cV_n(a))
		\ge (a^d\Den E - \delta_n - 1) 2^d a^{nd},
	\end{equation}
	for infinitely-many integers $n\ge 0$. It is easy to adapt
	Lemma \ref{lem:Frostman} to an analogous statement about
	$\nu^n_\rho(E;\, a)$ for all choices of $a$ [not only $a=\e$]. That
	endeavor shows us that
	$\nu^n_d(E;a) \ge a^d\Den E - \delta_n - 1$
	for infinitely-many integers $n\ge 0$.
	Therefore, 
	\begin{equation}
		\limsup_{n\to\infty} \nu^n_d(E;a)>0,
	\end{equation}
	provided that we choose 
	$a>\max\{1\,,(\Den E)^{-1/d}\}$.
	This and \eqref{Dim:a} together imply \eqref{eq:DenDim} and hence the
	lemma.
\end{proof}

\section{Peaks of Brownian motion}

Consider the set of times, after time $t=\exp(\e)$ [say], at which the Brownian motion
has LIL-type behavior. That is, let us fix some parameter $\gamma>0$,
and consider the random set
\begin{equation}\label{eq:L_B}
	\cP_B(\gamma) := 
	\left\{ s\ge \e^\e:\ \frac{B(s)}{(2 s\log\log s)^{1/2}}
	\ge \gamma\right\}.
\end{equation}
We are using a notation that is consistent with that in
\eqref{LIL:X} and \eqref{L_X}, where
$g(x) := (2x\log_+\log_+ x)^{1/2}$ is the gauge function that comes
to us naturally from the standard law of the iterated logarithm for
Brownian motion. That is, the following:

\begin{proposition}[The law of the iterated logarithm]\label{pr:LIL1}
	With probability one: 
	\begin{enumerate}
		\item $\cP_B(\gamma)$ is unbounded a.s.\ when $\gamma\le 1$;
		\item $\cP_B(\gamma)$ is bounded a.s.\ when $\gamma> 1$.
	\end{enumerate}
\end{proposition}

Actually, the standard textbook form of the LIL refers only to the case that
$\gamma<1$ and $\gamma>1$. The critical case $\gamma=1$ follows from
Motoo's work \cite[Example 2]{Motoo}. The following shows that when
$\cP_B(\gamma)$ is unbounded, that is when $\gamma \le1$,
it  is a macroscopic fractal of dimension one.
[Of course, $\cP_B(\gamma)$ does not have a 
remarkable macroscopic structure
when $\gamma>1$.]

\begin{theorem}\label{th:LIL:BM}
	Assertion \eqref{B:B} of Theorem \ref{B:OU} holds; that is,
	\begin{equation}
		\Dimh \cP_B(\gamma)=1
		\quad\text{a.s.\ for all $\gamma\le1$.}
	\end{equation}
\end{theorem}

In the subcritical case where $\gamma<1$, 
Theorem \ref{th:LIL:BM} follows immediately from
Lemma \ref{lem:Den:Dim} and the next statement.

\begin{proposition}[Strassen \cite{Strassen}]\label{pr:Strassen}
	Let $\nu$ denote the Lebesgue measure on $\R^d$.
	Then the following assertions are valid a.s.:
	\begin{enumerate}
		\item $\Den\cP_B(\gamma)>0$  when $\gamma<1$; and
		\item $\Den\cP_B(1)=0$.
	\end{enumerate}
\end{proposition}

\begin{proof}
	One can easily adapt a result of Strassen \cite[eq.\ (11)]{Strassen} about random
	walks to a statement about linear Brownian motion 
	in order to see that with probability one,
	\begin{equation}
		\Den\cP_B(\gamma) = 1-\exp\left\{ 
		-4\left(\frac{1}{\gamma^2}-1 \right)\right\},
	\end{equation}
	as long as $\gamma\le 1$. This does the job.
\end{proof}

Proposition \ref{pr:Strassen} shows that
Theorem \ref{th:LIL:BM} is interesting mostly in the critical case.
In the critical case, the random set
$\cP_B(1)$  is comprised of tall peaks of maximum possible 
asymptotic height. And Theorem \ref{th:LIL:BM} shows
that the set of tall peaks of critical height
has full dimension although it has zero density [Proposition
\ref{pr:Strassen}].

With the preceding remarks in mind, let us consider the following
random Borel measure that is supported in $\cP_B(1)$:
\begin{equation}\label{eq:Jn}
	\mu(G)
	:=\left\vert \cP_B(1) \cap G\right\vert
	\qquad\text{for all Borel sets $G\subseteq[4\,,\infty)$},
\end{equation}
where $\vert\,\cdots\vert$ denotes the
1-dimensional Lebesgue measure. 
The following is the key step in the proof of Theorem \ref{th:LIL:BM}.

\begin{proposition}\label{pr:LIL}
	$\sum_{n=4}^\infty\e^{-n} \mu(\cS_n) =\infty$ a.s.
\end{proposition}

We will begin our proof of Proposition \ref{pr:LIL} shortly. But first,
let us apply this proposition in order to establish Theorem \ref{th:LIL:BM}.

\begin{proof}[Proof of Theorem \ref{th:LIL:BM}]
	The preceding remarks tell us that we need to only consider the
	critical case, $\gamma=1$. Because $\mu[x\,,x+r) \le r,$ it follows that $K_1=1$,
	where $K_\rho$ was defined in \eqref{eq:K:rho}.
	Lemma \ref{lem:Frostman} and Proposition \ref{pr:LIL}
	together imply that
	\begin{equation}
		\sum_{n=4}^\infty\nu^n_1(\cP_B(1)\cap\cS_n)
		\ge \sum_{n=4}^\infty\e^{-n}\mu(\cS_n) =\infty
		\quad\text{a.s.}
	\end{equation}
	In particular, it follows that $\Dimh \cP_B(1)\ge 1$ a.s.,
	which is the desired result.
\end{proof}
 
 In order to derive Proposition \ref{pr:LIL}, let us consider the events
\begin{equation}
	\mathcal{E}_t := \left\{ \omega\in\Omega:\ B(t)(\omega) >
	(2 t\log\log t)^{1/2}\right\}\qquad\text{for all $t\ge 4$}.
\end{equation}

It is easy to see from l'H\^opital's rule that if $X$ has the standard
normal distribution, then $\P\{X> z\}$ is to within a multiplicative 
constant of $z^{-1}\exp(-z^2/2)$ uniformly for all $z\ge 1$.
The following is a consequence of this fact
and the strict positivity and the continuity of the Gaussian density function:
There exists a finite constant $c>1$ such that
\begin{equation}\label{eq:P(E)}
	\frac{1}{c\log t (\log\log t)^{1/2}}
	\le \P\left(\mathcal{E}_t\right) \le
	\frac{c}{\log t(\log\log t)^{1/2}}
	\qquad\text{for all $t\ge 4$}.
\end{equation}

Next we estimate $\P(\mathcal{E}_t\,|\,\mathcal{E}_s)$ for
various choices of $s<t$. The first quantifies the well-known
qualitative fact that $\mathcal{E}_s$ and $\mathcal{E}_t$
are approximately independent when $t\gg s$.

\begin{lemma}\label{lem:P(EcapE):far}
	For all real numbers $t>s>4$,
	\begin{equation}
		t>4s(\log\log s)(\log\log t)\quad\Rightarrow\quad
		\P(\mathcal{E}_t\,|\,\mathcal{E}_s) \le c \P(\mathcal{E}_t),
	\end{equation}
	where $c\in(0\,,\infty)$ does not depend on $(s\,,t)$.
\end{lemma}

\begin{proof}
	We recall the following well-known bound, which is essentially
	Lemma 1.5 of Orey and Pruitt \cite{OP}:
	If $U$ and $V$ are jointly distributed as a bivariate
	normal with common mean zero, common variance one,
	and covariance $\rho$, then there exists a finite
	constant $c$ such that
	\begin{equation}\label{eq:OP}
		\P(U>a\,|\,V>b)\le c\P\{U>a\}
		\qquad\text{whenever $|\rho|<(ab)^{-1}.$}
	\end{equation}
	Next we apply the preceding by setting $U:=t^{-1/2}B(t)$,
	$V:=s^{-1/2}B(s)$, $a:=(2\log\log t)^{1/2}$,  and
	$b:=(2\log\log s)^{1/2}$. Note that 
	$\rho=(s/t)^{1/2}$
	satisfies $0<\rho<(ab)^{-1}$ because $t>4s(\log\log s)(\log\log t)$.
	The lemma follows from \eqref{eq:OP}.
\end{proof}

When $t$ and $s$ are not too far apart, we do not
expect $\P(\mathcal{E}_t\,|\, \mathcal{E}_s)$ to have
the same order of magnitude as $\P(\mathcal{E}_t)$.
The following provides us with a quantitive estimate of 
$\P(\mathcal{E}_t\,|\, \mathcal{E}_s)$ in this case.

\begin{lemma}\label{lem:P(EcapE):close}
	There exists a finite constant $c$ such that for every $t>s>4$,
	\begin{equation}
		\P(\mathcal{E}_t\,|\, \mathcal{E}_s) \le 
		\frac{c}{(\log s)^{(t-s)/(4t)}}.
	\end{equation}
\end{lemma}

\begin{proof}
	According to Lemma 1.6 of Orey and Pruitt \cite{OP},
	if $U$ and $V$ are jointly distributed as a bivariate
	normal with common mean zero, common variance one,
	and covariance $\rho$, then there exists a finite
	constant $c$ such that
	\begin{equation}\label{eq:OP1}		
		\P(U>a\,|\,V>a) \le c\exp\left( -\frac18(1-\rho^2)a^2\right)
		\qquad\text{for all $a\ge 0$}.
	\end{equation}
	We apply this inequality with $U:=t^{-1/2}B(t)$,
	$V:=s^{-1/2}B(s)$, and $a:=(2\log\log s)^{1/2}$ to find that
	\begin{equation}
		\P(\mathcal{E}_t\,|\,\mathcal{E}_s) \le \P\left(B(t)
		> (2t\log\log s)^{1/2}\,\big|\, \mathcal{E}_s\right)
		=\P(U>a\,|\,V>a).
	\end{equation}
	Thus, the lemma follows from \eqref{eq:OP1}.
\end{proof}

We are  prepared to verify Proposition \ref{pr:LIL}.

\begin{proof}[Proof of Proposition \ref{pr:LIL}]
	For all $N\ge 4$ define
	\begin{equation}
		S_N := \sum_{n=4}^N \e^{-n} \mu(\cS_n).
	\end{equation}
	We intend to prove that $S_\infty:=\lim_{N\to\infty}S_N$ is
	infinite a.s. Because 
	$\E S_N = \sum_{n=4}^N \e^{-n}\int_{\e^{n-1}}^{\e^n}
	\P(\mathcal{E}_s )\,\d s$,
	we may apply \eqref{eq:P(E)} in order to see that
	\begin{equation}
		\E S_N \ge c_0(\log N)^{1/2}\qquad\text{for all $N\ge 4$,}
		\label{eq:LIL:ES_N}
	\end{equation}
	where $c_0$ is a positive constant that
	does not depend on $N$. Next, we estimate the variance of $S_N$.
	Clearly,
	\begin{equation}
		\E(S_N^2) \le Q_1+Q_2,
	\end{equation}
	where
	\begin{equation}\begin{split}
		Q_1&:=2\sum_{n=4}^N \e^{-2n}\iint_{\e^{n-1}<s<t<\e^n}
			\P(\mathcal{E}_s \cap \mathcal{E}_t )\,\d s\,\d t,\\
		Q_2&:= 2\mathop{\sum\sum}_{4\le n< m\le N} \e^{-n-m}
			\int_{\e^{n-1}}^{\e^n}\d s\int_{\e^{m-1}}^{\e^m}\d t\
			\P(\mathcal{E}_s \cap \mathcal{E}_t ).
		\label{eq:Q1Q2}
	\end{split}\end{equation}
	The elementary bound $\P(\mathcal{E}_s\cap\mathcal{E}_t)\le
	\P(\mathcal{E}_s)$ yields
	\begin{equation}\label{eq:LIL:Q1}
		Q_1\le  2\E S_N.
	\end{equation}
	
	We estimate $Q_2$ by splitting the double sum according to how much
	the summation variable $m$ is greater than the summation variable $n$.
	Before we hash out the details, let us first note that,
	according to Lemma \ref{lem:P(EcapE):close},
	whenever $m>n\ge 4$,
	\begin{equation}\begin{split}
		&\int_{\e^{n-1}}^{\e^n}\d s\int_{\e^{m-1}}^{\e^m}\d t\
			\P(\mathcal{E}_s \cap \mathcal{E}_t )\\
		&\hskip.5in\le c_1\int_{\e^{n-1}}^{\e^n}\d s\int_{\e^{m-1}}^{\e^m}\d t\
			\P(\mathcal{E}_s)\exp\left(-
			\frac{t-s}{4t}\log\log s\right)\\
		&\hskip.5in\le c_1\int_{\e^{n-1}}^{\e^n}\d s\int_{\e^{m-1}}^{\e^m}\d t\
			\P(\mathcal{E}_s)\exp\left(-c_2
			\frac{t-\e^{m-1}}{\e^m}\log n\right)\\
		&\hskip.5in\le c_1 \e^m\int_{\e^{n-1}}^{\e^n}\d s\int_0^\infty\d t\
			\P(\mathcal{E}_s)\e^{-c_2 t\log n}
			= \frac{c_3 \e^m}{\log n}\int_{\e^{n-1}}^{\e^n}\P(\mathcal{E}_s)\,\d s,
	\end{split}\end{equation}
	where $c_1,c_2,c_3$ are finite and positive constants that do not
	depend on $(n\,,m)$. Consequently,
	\begin{equation}\begin{split}
		\mathop{\sum\sum}_{4\le n< m\le (n+\alpha\log n)\wedge N} &\e^{-n-m}
			\int_{\e^{n-1}}^{\e^n}\d s\int_{\e^{m-1}}^{\e^m}\d t\
			\P(\mathcal{E}_s \cap \mathcal{E}_t )\\
		&\hskip.5in
			\le c_3\alpha \sum_{n=4}^N \e^{-n}\int_{\e^{n-1}}^{\e^n}\P(\mathcal{E}_s)\,\d s
			= c_3\alpha\E S_N,
		\label{eq:Q21}
	\end{split}\end{equation}
	for all $\alpha>0$. We emphasize that $c_1$, $c_2$,
	and $c_3$ are finite constants that do not depend on $(N\,,\alpha)$,
	and $\alpha$ is, so far, an arbitrary parameter.
	
	Next we observe that there exists $\alpha>0$ large enough such that 
	for all integers $m,n\ge 4$,
	\begin{equation}
		m>n+\alpha\log n\qquad\Rightarrow
		\qquad\e^{m-1} > 4  \e^n (\log n)(\log m).
	\end{equation}
	We will choose $\alpha$ to be this particular value, both in the preceding
	and in what follows.
	In this case, it is then easy to see that
	$t>4s(\log\log s)(\log\log t)$
	for every $s\in\cS_n$ and $t\in\cS_m$,
	as long as $n,m\ge 4$ are arbitrary integers that
	satisfy $m>n+\alpha\log n$. Thus,
	Lemma \ref{lem:P(EcapE):far} ensures that, for this particular choice of
	$\alpha$,
	\begin{equation}\begin{split}
		&\mathop{\sum\sum}_{4\le n<  n+ \alpha \log n < m\le N} \e^{-n-m}
			\int_{\e^{n-1}}^{\e^n}\d s\int_{\e^{m-1}}^{\e^m}\d t\
			\P(\mathcal{E}_s \cap \mathcal{E}_t )\\
		&\hskip1.5in
			\le c_4\left( \sum_{n=4}^N \e^{-n}\int_{\e^{n-1}}^{\e^n}\P(\mathcal{E}_s)\,\d s
			\right)^2
			= c_4\left|\E S_N\right|^2,
		\label{eq:Q22}
	\end{split}\end{equation}
	where $c_4$ is a finite constant that does not depend on $n$.
	Thanks to \eqref{eq:Q1Q2}, \eqref{eq:Q21}, and \eqref{eq:Q22}, it follows that 
	$Q_2\le c_3\alpha\E S_N+c_4|\E S_N|^2$, uniformly in all $N\ge 4$.
	Thus, it follows from \eqref{eq:LIL:ES_N} and \eqref{eq:LIL:Q1}
	that 
	\begin{equation}\label{eq:BM:UB}
		\E\left(S_N^2\right) =O\left( \left|\E S_N\right|^2\right)
		\qquad\text{as $N\to\infty$}.
	\end{equation}
	Since $\E S_N\uparrow\infty$ as $N\to\infty$ [see \eqref{eq:LIL:ES_N}]
	and $S_\infty\ge S_N$ for all $N$,
	\begin{equation}
		\P\left\{S_\infty=\infty\right\}
		\ge\liminf_{N\to\infty}\P\left\{ S_N\ge\tfrac12\E S_N\right\}
		\ge\frac{1}{4}\liminf_{N\to\infty}\frac{|\E S_N|^2}{\E(S_N^2)}
		>0,
	\end{equation}
	thanks to the Paley--Zygmund inequality \cite[Lemma $\gamma$]{PZ},
	which states that if $Z$ is in $L^2(\P)$ and $\|Z\|_\infty>0$, then
	\begin{equation}
		\P\left\{Z> \tfrac12\E Z\right\}\ge \frac 1 4 \frac{(\E Z)^2}{\E(Z^2)}.
	\end{equation} 
	An appeal to the Hewitt--Savage 0--1 law  completes the proof.
\end{proof}

\section{General bounds}\label{sec:Gen}

Let $X:=\{X_t\}_{t\in T}$ be a real-valued random field
with continuous trajectories, where $T\subseteq\R^d$
is either one of the $2^d$ standard closed orthants of $\R^d$,
or $T$ is $\R^d$ itself.

For all real numbers $b\in(0\,,\infty)$
we can define
\begin{equation}\label{eq:gen:tail:LB}
	\bm{c}(b) := -\limsup_{z\to\infty} z^{-b}
	\sup_{t\in T}\log\P\left\{ X_t > z\right\},
\end{equation}
and
\begin{equation}\label{eq:gen:tail:UB}
	\bm{C}(b):=
	-\liminf_{z\to\infty} z^{-b}\inf_{t\in T}\log\P\left\{ X_t> z\right\}.
\end{equation}
Of course, $0\le \bm{c}(b)\le \bm{C}(b)\le\infty$ for all $b>0$.

Define $|t|$ to be the $\ell^\infty$-norm of $t\in\R^d$; that is,
\begin{equation}
	|t| := \max_{1\le i\le d} |t_i|\qquad\text{for
	all $t:=(t_1\,,\ldots,t_d)\in\R^d$}.
\end{equation}
Then it should be intuitively clear that, under mild conditions on large-scale
smoothness and asymptotic pairwise independence of $X$,
one ought to be able to prove that
\begin{equation}\label{eq:LIL}
	0<\limsup_{|t|\to\infty} \frac{X_t}{(\log |t|)^{1/b}}
	<\infty\quad\text{a.s.,}
\end{equation}
provided additionally that $0<\bm{c}(b)\le \bm{C}(b)<\infty$ 
for a certain special value of $b\in(0\,,\infty)$.
In other words, we might expect that
if $0<\bm{c}(b)\le \bm{C}(b)<\infty$, then the tall peaks of the process
$X$ are typically gauged, to within a constant, by the function 
$t\mapsto(\log |t|)^{1/b}$.

The main results of this section are two general macroscopic
Hausdorff dimension estimates; see Theorems
\ref{th:Gen:LIL:UB} and \ref{th:Gen:LIL:LB} below. The first
theorem describes conditions, similar to those outlined earlier,
which ensure  the upper bound in \eqref{eq:LIL}, and also
bound from above the macroscopic Hausdorff dimension of the set of
times that $X_t$ exceeds a [correct] constant multiple of $(\log |t|)^{1/b}$.
The second theorem turns out to be a much more subtle result that
produces matching lower bounds for the dimension of these exceedance
times. Thus, we begin with a general upper bound for the Hausdorff dimension
of the tall peaks of a stochastic process.

\begin{theorem}[A general upper bound]\label{th:Gen:LIL:UB}
	Suppose that there exists $b\in(0\,,\infty)$
	such that $\bm{c}(b)>0$ and for all $\gamma\in(0\,,d)$,
	\begin{equation}\label{cond:UB}
		\sup_{w\in T}\P\left\{\sup_{t\in[w,w+1)}X_t > 
		\left( \frac{\gamma}{\bm{c}(b)}\,\log s\right)^{1/b}\right\}
		\le s^{-\gamma + o(1)}\text{ as $s\to\infty$}.
	\end{equation}
	Then,
	\begin{equation}\label{gen:LIL:UB}
		\limsup_{|t|\to\infty} \frac{X_t}{(\log |t|)^{1/b}}
		\le \left(\frac{d}{\bm{c}(b)}\right)^{1/b} \qquad\text{a.s.}
	\end{equation}
	Furthermore, 
	\begin{equation}\label{dim:gen:UB}
		\Dimh
		\left\{t\in T:\ |t|>\exp(\e),\
		X_t \ge \left(\frac{\gamma}{\bm{c}(b)}\log |t|\right)^{1/b}\right\}\le
		d-\gamma,
	\end{equation}
	for all $\gamma\in(0\,,d)$.
\end{theorem}

\begin{proof}
	It suffices to prove the result in the case that $T=\R^d_+$. The
	other cases, including $T=\R^d$, follow from this after making small
	adjustments.
	
	The stated $\limsup$ result is a more-or-less standard 
	exercise in the Borel--Cantelli lemma, and the upper bound
	on the dimension follows from a routine covering argument,
	the likes of which are familiar for bounding the microscopic Hausdorff
	of a random set. We include the proof for the sake of completeness.
	
	Let us write $\Dimh \Sigma<0$ when $\Sigma$ is a bounded set.
	We plan to prove that \eqref{dim:gen:UB}
	holds for all $\gamma>0$; \eqref{gen:LIL:UB} follows immediately
	from this formulation of \eqref{dim:gen:UB}. From now on we choose and
	fix an arbitrary $\gamma>0$.
	
	Our goal is to prove that
	\begin{equation}\label{gen:DimL:UB}
		\Dimh \Lambda_u \le d-\gamma\qquad\text{a.s.,}
	\end{equation}
	where 
	\begin{equation}
		\Lambda_u :=\left\{t\in \R^d_+:\ |t|>\exp(\e),\
		X_t \ge \left(\frac{\gamma}{\bm{c}(b)}\log |t|\right)^{1/b}\right\}.
	\end{equation}
	Condition \eqref{cond:UB} ensures that for all $\varepsilon\in(0\,,\gamma)$
	there exists a finite constant $K_\varepsilon$
	such that for all $m:=(m_1\,,\ldots,m_d)\in\R^d_+$ 
	that satisfy $\|m\|>\exp(\e)$,
	\begin{equation}\begin{split}
		\P\{\Lambda_u\cap Q(m\,,1)\neq\varnothing\}
			&\le \P\left\{ \sup_{t\in Q(m,1)} X_t \ge
			\left(\frac{\gamma}{\bm{c}(b)}\log |m|\right)^{1/b}\right\}\\
		&\le \frac{K_\varepsilon}{|m|^{\gamma-\varepsilon}}.
	\end{split}\end{equation}
	Therefore, we can cover $\Lambda_u\cap\cS_n$ by upright boxes of sidelength
	$r\equiv 1$ in order to see that for all $\rho>0$ and $n\ge 0$ large,
	\begin{equation}\begin{split}
		\E\left[\nu^n_\rho\left(\Lambda_u\right) \right]
			&\le \sum_{\substack{m\in\Z^d_+:\\
			Q(m,1)\subset\cS_n}}\left(\frac{1}{\e^n}\right)^\rho
			\P\left\{\Lambda_u\cap Q(m\,,1) \neq \varnothing\right\}\\
		&\le K_\varepsilon \e^{-n\rho} \sum_{\substack{m\in\Z^d_+:\\
			Q(m,1)\subset\cS_n}}|m|^{-\gamma+\varepsilon}.
	\end{split}\end{equation}
	Whenever $n\in\mathbf{N}$ and $m\in\Z^d_+$ are such that
	$Q(m\,,1)$ lies entirely in $\cS_n$, then it must be
	that $|m| \ge \exp(n-1)$.
	Since $\cS_n$ contains at most $\text{const}\cdot \exp(nd)$ upright boxes
	of sidelength one, it follows that
	\begin{equation}
		\E\left[\nu^n_\rho(\Lambda_u)\right] \le \text{const}\cdot\e^{
		-n[\rho+\gamma-\varepsilon-d]},
	\end{equation}
	for all $n\ge 1$ sufficiently large. In particular,
	\begin{equation}
		\E\left[\sum_{n=0}^\infty\nu^n_\rho(\Lambda_u)\right]<\infty
		\qquad\text{if $\rho>d-\gamma+\varepsilon$}.
	\end{equation}
	This proves that $\Dimh\Lambda_u\le \rho$ a.s.\
	for all $\rho>d-\gamma+\varepsilon$. Send 
	$\rho\downarrow d-\gamma+\varepsilon$
	and then $\varepsilon\downarrow0$, in this order,
	to deduce \eqref{gen:DimL:UB} and hence the theorem.
\end{proof}

We now move on to a perhaps more interesting
study of lower bounds for $\Dimh$ of high peaks of $X$;
more specifically,
our next result shows that one can sometimes obtain good lower bounds
on the macroscopic dimension in the statement of Theorem \ref{th:Gen:LIL:UB}.

Consider, for every $\gamma\in(0\,,d)$, the random set
\begin{equation}\label{eq:L_X}
	\Lambda_\ell:= 
	\left\{t\in T:\ |t|>\exp(\e),\
	X_t \ge \left(\frac{\gamma}{\bm{C}(b)}\log t\right)^{1/b}\right\}.
\end{equation}
We plan to show that, under some conditions on the process $X$,
\begin{equation}\label{gen:DimL:LB}
	\Dimh\Lambda_\ell\ge d-\gamma\qquad\text{a.s.,}
\end{equation}
thus obtaining a complimentary bound to that of Theorem \ref{th:Gen:LIL:UB}.

The standard way to obtain lower bounds on the ordinary Hausdorff dimension
of a set is to find a smooth measure on that set;
see Lemma \ref{lem:Frostman} for example. The said smooth measure is typically
``uniquely canonical,'' and readily guessed when the set in question has good local
structure. The remaining work is in determining the exact order of the smoothness of the
canonical measure.

In the present setting of macroscopic dimension, 
one can in fact prove that the random set $\Lambda_\ell$ does \emph{not} support a 
``uniquely canonical'' measure. Instead we
have to use a different sort of argument.

Before we proceed
we need to develop a few general facts about macroscopic Hausdorff dimension.
We will also introduce some notation that will be used throughout the remainder of this
section.

\begin{definition}\label{def:skeleton}
	Let $\Pi_n$ be a finite collection of points in $\R^d$
	for every integer $n\ge 0$. Given a real number $\theta\in(0\,,1)$,
	we say that $\{\Pi_n\}_{n=0}^\infty$ is a \emph{$\theta$-skeleton}
	of $\R^d$ if there exists an integer $N=N(\theta)$ such that:
	\begin{enumerate}
		\item For every $n\ge N$,
			\begin{equation}
				\bigcup_{x\in\Pi_n} Q( x\,,\e^{\theta n})\subset\cS_n;
			\end{equation}
		\item If $x$ and $y$ are distinct points in $\Pi_n$ for
			some $n\ge N$, then
			\begin{equation}
				Q(x\,,\e^{\theta n})\cap Q(y\,,\e^{\theta n})=\varnothing;
				\quad\text{and}
			\end{equation}
		\item There exists a constant $a=a(d\,,\theta)\in(0\,,1)$ such that
			\begin{equation}\label{eq:Ln}
				a\e^{nd(1-\theta)} \le |\Pi_n| \le a^{-1} \e^{nd(1-\theta)},
			\end{equation}
			where ``$|\,\cdots|$'' denotes cardinality.
	\end{enumerate}
\end{definition}

Given some $\theta\in(0\,,1)$, 
$\R^d$ has uncountably-many $\theta$-skeletons. 
From now on, we choose and fix one such choice,
and denote it by $\Pi(\theta):=\{\Pi_n(\theta)\}_{n=0}^\infty$. 
For instance, we might wish to opt for the following construction, to be explicit:
\begin{equation}
	\Pi_n(\theta) := A_n(\theta)\times\cdots\times A_n(\theta)
	\qquad[d\text{ times}];
\end{equation}
where
\begin{equation}
	A_n(\theta) := \bigcup_{\substack{%
	0\le j\le \e^{n(1-\theta)+1}-\e^n:\\ j\in\Z}}
	\left\{ \e^n + j \e^{\theta n} \right\}
\end{equation}
Other constructions are also possible.
The property of $\Pi(\theta)$ that we are after this: $\Pi_n(\theta)$ is basically
a ``near-optimal $\e^{\theta n}$-packing'' of $\cS_n$
for all large $n$. Part 2 of Definition
\ref{def:skeleton} codifies the precise notion of ``packing''
and Part 3 makes precise our ``near-optimal'' sense. 

\begin{definition}\label{def:thick}
	Let $E\subseteq\R^d$ be a set and choose and fix
	some real number $\theta\in(0\,,1)$. We say that
	$E$ is \emph{$\theta$-thick} if there exists 
	an integer $M=M(\theta)$ such that 
	\begin{equation}
		E\cap Q(x\,,\e^{\theta n})\neq\varnothing,
	\end{equation}
	for all $x\in\Pi_n(\theta)$ and $n\ge M$.
\end{definition}

In words, $\theta$-thick sets are ``approximately self-similar
sets.''

We are ready to present one of the novel technical statements of
this section.

\begin{proposition}\label{pr:thick}
	If $E\subset\R^d$ is $\theta$-thick for some 
	$\theta\in(0\,,1)$, then 
	\begin{equation}
		\Dimh E\ge d(1-\theta).
	\end{equation}
\end{proposition}

Proposition \ref{pr:thick} presents us with a strategy
for obtaining a lower bound for $\Dimh F$ that can be
different from a Frostman-type method [Lemma \ref{lem:Frostman}]. The following 
is an immediate consequence of Proposition \ref{pr:thick}.

\begin{corollary}\label{co:thick}
	If $F\subset\R^d$ has a $\theta$-thick subset $E$ for some $\theta\in(0\,,1)$,
	then $\Dimh F\ge d(1-\theta)$.
\end{corollary}
 
 
 
In principle, our definition of $\theta$-thickness depends on
our {\it a priori} choice of a $\theta$-skeleton $\Pi(\theta)$. We are not aware
of any example where the choice matters very much.  But even if the
choice did matter, we can interpret Corollary \ref{co:thick} as saying that 
we can always obtain a lower bound
on $\Dimh F$ by finding a $\theta$-skeleton $\Pi(\theta)$ of $\R^d$
and a set $E\subset F$ that is $\theta$-thick with respect to our choice of skeleton.
In that case, $\Dimh F\ge d(1-\theta)$.

 It remains to prove Proposition \ref{pr:thick}.
 
\begin{proof}[Proof of Proposition \ref{pr:thick}]
	Using the notation of Definitions \ref{def:skeleton} and \ref{def:thick} we can find
	a finite number of points $x_{1,n},\ldots,x_{L_n,n}\in\cS_n$ such that
	\begin{equation}
		|x_{i,n}-x_{j,n}| \ge \e^{\theta n}
		\qquad\text{for all $1\le i\neq j\le L_n$,}
	\end{equation}
	where in fact $L_n := L_n(\theta) := |\Pi_n(\theta)|$. In particular,
	\eqref{eq:Ln} ensures that
	\begin{equation}
		a\e^{nd(1-\theta)}\le L_n\le a^{-1} \e^{nd(1-\theta)} \qquad
		\text{for all $n$}.
	\end{equation}
	It might help to recall that
	$a\in(0\,,1)$ is independent of the parameters $(i\,,j\,,n)$ of this discussion.
	
	Even though $F$ might not support a particularly-natural
	measure, the following defines a very natural
	locally finite measure $\mu$ on $E$:
	\begin{equation}
		\mu(F) := \sum_{n=M}^\infty\sum_{j=1}^{L_n}
		\1_F(x_{j,n}),
	\end{equation}
	for all $F\subseteq\R^d$. 
	
	Because $\mu(\cS_n)  = L_n$, the total $\mu$-mass of the $n$th shell satisfies
	\begin{equation}\label{eq:mu:0}
		a\e^{nd(1-\theta)} \le \mu(\cS_n) \le a^{-1}\e^{nd(1-\theta)}
		\qquad\text{for all $n\ge M$}.
	\end{equation}
	Since $\mu$ is a measure on $E$, we plan to use the measure
	$\mu$ in Lemma \ref{lem:Frostman} in order to find a lower bound
	for $\Dimh(E)$. With this aim in mind,
	we next establish an upper bound for $\mu(Q)$ for 
	every uprightbox $Q$ whose sidelength is at least one, with the 
	sole provision that $Q\subset\cS_n$ for some $n\ge M$ that is sufficiently
	large.
	
	Let us consider an arbitrary upright box $Q(z\,,r)$ of sidelength $r\ge 1$
	such that $Q(z\,,r)$ lies entirely in $\cS_n$ for some large enough integer $n\ge M$.
	Elementary properties of the Euclidean space $\R^d$ imply that
	there exists a positive integer $k\le 2^d$ together with $k$ points
	$z_1,\ldots,z_k$ from the collection $\{x_{1,n},\ldots,x_{L_n,n}\}$
	such that 
	\begin{equation}
		Q(z\,,r) \subseteq \bigcup_{j=1}^k Q(z_j\,,r).
	\end{equation}
	Therefore,
	\begin{equation}
		\mu(Q(z\,,r)) \le \sum_{j=1}^k \mu(Q(z_j\,,r)) = k \le 2^d.
	\end{equation}
	This shows that for all $\rho>0$, 
	\begin{equation}\label{gen:K_rho:1}
		K_\rho := \sup\left\{ \frac{\mu(Q)}{r^\rho}:\,
		Q\in\mathcal{B},\ Q\subset\cS_n,\ {\rm side}(Q)\in[1\,,r],\ r\ge 1\right\}
		\le\ 2^d.
	\end{equation}
	On the other hand, if $r\ge\e^{\theta n}$, then an upright box
	$Q(z\,,r)$ can contain at most $(1+r\e^{-\theta n})^d$-many points
	in $\Pi_n(\theta)$. Now, 
	\begin{equation}
		(1+r\e^{-\theta n})^d\le 2^d r^d \e^{-d\theta n},
	\end{equation}
	because $r\ge\e^{\theta n}$, and ${\rm side}(Q(z\,,r))\le\e^n$ because
	$Q(z\,,r)\subset\cS_n$. Therefore,
	it follows from the construction of the measure $\mu$ that
	\begin{equation}\begin{split}
		&\sup\left\{ \frac{\mu(Q)}{[{\rm side}(Q)]^\rho}:\,
			Q\in\mathcal{B},\ Q\subset\cS_n,\ 
			{\rm side}(Q)\ge \e^{\theta n}\right\}\\
		&\hskip1.5in\le 2^d \e^{-d\theta n}\sup_{\e^n\ge r\ge \e^{\theta n}} r^{d-\rho}
			= 2^d \e^{-n[d\theta-d+\rho]}\label{gen:K_rho:2}
	\end{split}\end{equation}
	as long as $0<\rho<d$. Now we compare \eqref{gen:K_rho:1}
	and \eqref{gen:K_rho:2}, and set $\rho:=d(1-\theta)$
	in order to see that $K_{d(1-\theta)}\le 2^d$, where  $K_\rho$
	was defined in \eqref{gen:K_rho:1}; see also
	\eqref{eq:K:rho}. This bound
	and \eqref{eq:mu:0} together yield the following: For all $n\ge M$ sufficiently large,
	\begin{equation}
		\nu^n_{d(1-\theta)} (E) \ge 2^{-d} \e^{-nd(1-\theta)}\mu(\cS_n) 
		\ge  a2^{-d},
	\end{equation}
	thanks to Lemma \ref{lem:Frostman}. It follows that 
	$\limsup_{n\to\infty}\nu^n_{d(1-\theta)}(E)\ge a2^{-d}>0, $
	and hence
	\begin{equation}
		\Dimh(E)\ge \lDimh E\ge d(1-\theta),
	\end{equation}
	where $\lDimh$ denotes the
	lower Hausdorff dimension of Barlow and Taylor \cite{BarlowTaylor,BarlowTaylor1},
	as was recalled in \eqref{lDimh}.
\end{proof}

Let us introduce a last piece of notation before we state
and prove the general lower bound of this section [Theorem
\ref{th:Gen:LIL:LB} below]. That lower bound
will be a counterpart to Theorem \ref{th:Gen:LIL:UB}.
\begin{definition}
	Let $\mathcal{I}$ denote the collection of all independent
	finite sequences of independent random variables.
\end{definition}

Then we have the following general lower bound statement.

\begin{theorem}\label{th:Gen:LIL:LB}
	Suppose there exists $b\in(0\,,\infty)$ such that $\bm{C}(b)<\infty$.
	Suppose in addition that there exist $\delta\in(0\,,1)$ and an increasing nonrandom 
	measurable function $S:\R\to\R$ such that
	\begin{equation}\label{cond:LB}
		n^{-1} \max_{\{t_i\}_{i=1}^m\in\Pi_n(\delta)}
		\max_{1\le j\le m}
		\inf_{\{Y_i\}_{i=1}^m\in\mathcal{I}}\log\P \{ |S(X_{t_j})
		- S(Y_j)|> 1\} \to -\infty,
	\end{equation}
	as $n\to\infty$.
	Then, 
	\begin{equation}\label{eq:Gen:LIL:LB}
		\limsup_{t\to\infty} \frac{X_t}{(\log t)^{1/b}}\ge \left( \frac{d}{\bm{C}(b)}
		\right)^{1/b},
	\end{equation}
	a.s. Moreover, if $\gamma\in(0\,,d)$ then
	\begin{equation}\label{eq:Gen:Dim:LB}
		\Dimh \left\{t\in T:\ |t|>\exp(\e),\
		\frac{X_t}{(\log t)^{1/b}} \ge 
		\left(\frac{\gamma}{\bm{C}(b)}\right)^{1/b}\right\}\ge
		d-\gamma\ \text{a.s.}
	\end{equation}
	
\end{theorem}

\begin{remark} 
	Condition \eqref{cond:LB} is a coupling assumption, and states that if
	$t_1,\ldots,t_m\in T$ have large norms [say, are in $\cS_n$ for a large $n$]
	and sufficiently far apart [say at least $\exp(\delta n)$ apart],
	then $X_{t_1},\ldots,X_{t_m}$
	are close---say within one unit---to an independent sequence with very
	high probability.
	At first glance this might seem to be a technical and complicated condition.
	We will see in the next few sections that \eqref{cond:LB}
	is in fact frequently easy to use, particularly in the context of stochastic
	PDEs.
	Condition \eqref{cond:LB} can be recast in terms of the ``correlation
	length'' of the process $X$; see Conus et al \cite{CJK-islands} for details.
\end{remark}

If $\Dimh G>0$ then in particular $G$ is unbounded. In this way we see that
\eqref{eq:Gen:Dim:LB} implies that 
\begin{equation}
	\limsup_{t\to\infty}\frac{X_t}{(\log t)^{1/b}}
	\ge \left(\frac{\gamma}{\bm{C}(b)}\right)^{1/b},
\end{equation}
a.s.\ for all $\gamma\in(0\,,d)$.
Let $\gamma\uparrow d$ to deduce \eqref{eq:Gen:LIL:LB} from \eqref{eq:Gen:Dim:LB}. 
Thus, we need to  derive only \eqref{eq:Gen:Dim:LB}.

\begin{proof}[Proof of Theorem \ref{th:Gen:LIL:LB}]
	Since $X_t\ge \alpha$ if and only if $S(X_t)\ge S(\alpha)$,
	we can replace the random field $\{X_t\}_{t\in T}$
	by the random field $\{S(X_t)\}_{t\in T}$ throughout
	the entire statement of the theorem in order to see that
	for the remainder of the proof
	we can---and will---assume
	without incurring any loss in generality that 
	\begin{equation}
		S(x):=x\qquad\text{for all $x\in\R$.}
	\end{equation}
	In other words, the function $S$ merely plays the role of a change
	of ``scale.''

	We plan to prove that the random set 
	$\Lambda_\ell$---defined earlier
	in \eqref{eq:L_X}---a.s.\ is $\theta$-thick  for every 
	$\theta\in(\gamma/d\,,1)$ and $\gamma\in(0\,,d)$.
	Owing to Proposition \ref{pr:thick}, this proves that 
	\begin{equation}
		\Dimh \Lambda_\ell\ge d(1-\theta)\qquad\text{%
		a.s.\ for all $\theta\in\left(\frac{\gamma}{d}\,,1\right)$ and $\gamma\in(0\,,d)$,}
	\end{equation}
	and \eqref{eq:Gen:Dim:LB} follows. In light of the paragraph that followed the
	statement of Theorem \ref{th:Gen:LIL:LB}, this endeavor completes the proof
	of Theorem \ref{th:Gen:LIL:LB}. Henceforth, we choose and fix two arbitrary
	numbers $\gamma\in(0\,,d)$ and
	$\theta\in(\gamma/d\,,1)$. We also hold fixed an arbitrary [small]
	\begin{equation}\label{eq:delta}
		0<\delta<\theta-\frac{\gamma}{d}.
	\end{equation}
	
	Now we carry out a multi-scale argument. Recall the definition of
	$\Pi_n(\theta)$. Because of that definition,
	for all sufficiently large integers $n\gg1$, we can find points
	$x_{1,n},\ldots,x_{L_n,n}$ in $\cS_n$ such that
	\begin{equation}
		Q(x_{i,n}\,,\e^{\theta n})
		\cap Q(x_{j,n}\,,\e^{\theta n})=\varnothing
		\qquad\text{when $1\le i\ne j\le L_n$},
	\end{equation}
	and
	\begin{equation}
		a\e^{nd(1-\theta)}\le L_n\le a^{-1}\e^{nd(1-\theta)},
	\end{equation}
	where $a\in(0\,,1)$ depends neither on $n$ nor on the pair $(i\,,j)$.
	For future purposes, we would like to emphasize that as part of the construction
	of these points we have also the following:
	\begin{equation}\label{nlogL}
		\lim_{n\to\infty}\frac{\log L_n}{n}=d(1-\theta).
	\end{equation}
	
	For all $1\le i\le L_n$
	we can find points $z_{1,n}(i),\ldots,z_{\ell_n(i),n}(i)$ in
	$Q(x_{i,n}\,,\e^{\theta n})$ such that 
	\begin{equation}
		|z_{k,n}(i)-z_{l,n}(i)|>\e^{\delta n},
	\end{equation}
	whenever $1\le k\neq l\le \ell_n(i)$, and
	\begin{equation}\label{l_n(i)}
		b\e^{nd(\theta-\delta)}\le \ell_n(i) \le b^{-1}\e^{nd(\theta-\delta)},
	\end{equation}
	where $b\in(0\,,1)$ depends neither on $n$ nor on the triple $(i\,,k\,,l)$.
	In fact, $a$ and $b$ depend only on $\theta$ and $\delta$, which
	are held fixed throughout this discussion.
	
	According to \eqref{cond:LB}
	for all $K>d$ and  $1\le i\le L_n$
	we can find an independent sequence $Y_1,\ldots,Y_m$ of random variables
	such that
	\begin{equation}\label{babove}
		\max_{1\le j\le \ell_n(i)}
		\P\{|X_{z_{j,n}(i)}-Y_j|>1\}\le K\e^{-Kn}\qquad\text{for all $n\ge K$}.
	\end{equation}
	The particular construction of $Y_1,\ldots,Y_m$ might---or might not---depend
	on $K$ and $i$; it does not matter. The upshot is the following: Since
	\begin{equation}\begin{split}
		&\P\left\{\sup_{t\in Q(x_{i,n}\, ,\e^{\theta n})} \frac{X_t}{(\log t)^{1/b}} \le
			\left(\frac{\gamma}{\bm{C}(b)}\right)^{1/b}\right\}\\
		&\hskip1in\le \P\left\{\max_{1\le j\le \ell_n(i)} X_{z_{j,n}(i)} \le
			\left(\frac{(n+1)\gamma}{\bm{C}(b)}\right)^{1/b}\right\},
	\end{split}\end{equation}
	two back-to-back applications of \eqref{babove} show us that
	\begin{align}\notag
		&\P\left\{\sup_{t\in Q(x_{i,n}\,\e^{\theta n})} \frac{X_t}{(\log t)^{1/b}} \le
			\left(\frac{\gamma}{\bm{C}(b)}\right)^{1/b}\right\}\\
		&\le K\e^{-Kn} + \prod_{j=1}^{\ell_n(i)}\P\left\{ Y_j \le 1+
			\left(\frac{(n+1)\gamma}{\bm{C}(b)}\right)^{1/b}\right\}\\
		&\le K\e^{-Kn} + \prod_{j=1}^{\ell_n(i)}\left( \P\left\{ X_{z_{j,n}(i)}\le 2+
			\left(\frac{(n+1)\gamma}{\bm{C}(b)}\right)^{1/b}\right\} + K\e^{-Kn} \right),
	\end{align}
	uniformly for all $0\le i\le L_n$ and $n\ge K$. Given an arbitrary
	$\varepsilon\in(0\,,1)$, we can find $K_0>K$ such that
	\begin{align}\notag
		\P\left\{ X_{z_j(i)}\le 2+ \left(\frac{(n+1)\gamma}{\bm{C}(b)}\right)^{1/b}\right\} 
			&\le\P\left\{ X_{z_j(i)}\le 
			\left(\frac{n(\gamma+\varepsilon)}{\bm{C}(b)}\right)^{1/b}\right\} \\
		&\le 1 - \e^{-(\gamma+\varepsilon)n},
	\end{align}
	uniformly for all $0\le i\le L_n$, $0\le j\le \ell_n(i)$, and $n\ge K_0$.
	This follows because $\bm{C}(b)<\infty$. 
	
	Because of \eqref{eq:delta},
	we can---and will---in fact choose $\varepsilon$
	small enough so that 
	\begin{equation}\label{eq:eps}
		0<\varepsilon<\frac{ d(\theta-\delta)-\gamma}{2}.
	\end{equation}
	Therefore, there exists $K_1>K_0$ such that
	\begin{align}
		&\P\left\{\sup_{t\in Q(x_{i,n},\e^{\theta n})} \frac{X_t}{(\log t)^{1/b}} \le
			\left(\frac{\gamma}{\bm{C}(b)}\right)^{1/b}\right\}\\\notag
		&\hskip2in
			\le K\e^{-Kn} + \left( 1-\e^{-(\gamma+\varepsilon)n} 
			+ K\e^{-Kn} \right)^{\ell_n(i)}\\\notag
		&\hskip2in
			\le K \e^{-Kn} + \exp\left\{ -\ell_n(i)\e^{-(\gamma+2\varepsilon)n}\right\},
	\end{align}
	uniformly for all $0\le i\le L_n$ and $n\ge K_1$.  [The preceding hinges
	on the fact that $K>d>\gamma+\varepsilon$.]
	We may deduce from \eqref{l_n(i)} that
	\begin{align}
		&\P\left\{\sup_{t\in Q(x_{i,n},\e^{\theta n})} \frac{X_t}{(\log t)^{1/b}} \le
			\left(\frac{\gamma}{\bm{C}(b)}\right)^{1/b}\right\}\\\notag
		&\hskip2in
			\le K \e^{-Kn} + \exp\left\{ -b\e^{(d\theta 
			-d\delta -\gamma-2\varepsilon)n}\right\},
	\end{align}
	uniformly for all $0\le i\le L_n$ and $n\ge K_1$.
	Thanks to \eqref{eq:eps} and the fact that 
	$K>d>n^{-1}\log L_n = d(1-\theta)+o(1)$---see \eqref{nlogL}---it
	follows that
	\begin{equation}
		\sum_{n=1}^\infty\sum_{i=0}^{L_n}
		\P\left\{\sup_{t\in Q(x_{i,n},\e^{\theta n})} \frac{X_t}{(\log t)^{1/b}} \le
		\left(\frac{\gamma}{\bm{C}(b)}\right)^{1/b}\right\}<\infty.
	\end{equation}
	Therefore, the Borel--Cantelli lemma ensures that the following holds
	for all $\omega$ in the probability space that lie
	outside of a single set of $\P$-measure zero: For all but a finite number
	of integers $n\ge 1$,
	\begin{equation}
		\sup_{t\in Q(x_{i,n},\e^{\theta n})} \frac{X_t(\omega)}{(\log t)^{1/b}} >
		\left(\frac{\gamma}{\bm{C}(b)}\right)^{1/b}
		\qquad\text{for all $0\le i\le L_n$}.
	\end{equation}
	Recall the random set $\Lambda_\ell$ that was defined earlier in
	\eqref{eq:L_X}. We can state the preceding display in another way;
	namely, that for all but a finite number of integers $n\ge 1$,
	\begin{equation}
		\Lambda_\ell\cap Q(x_{i,n},\e^{\theta n})\neq\varnothing
		\qquad\text{for all $0\le i\le L_n$}.
	\end{equation}
	This proves that $\Lambda_\ell$ is $\theta$-thick a.s. Proposition \ref{pr:thick}
	then shows that 
	\begin{equation}
		\Dimh\Lambda_\ell\ge d(1-\theta)\qquad\text{a.s.}
	\end{equation}
	Since $\theta\in(\gamma/d\,,1)$ were arbitrary, 
	we  let $\theta\downarrow \gamma/d$
	to complete the proof.
\end{proof}

Let us close this section by recalling a well-known general asymptotic evaluation
of the tail of the distribution of the
supremum of a stationary Gaussian process with a nice covariance
function. The result is originally due to Pickands \cite[Lemma 2.5]{Pickands},
with extra conditions that were removed subsequently by 
Qualls and Watanabe \cite[Theorem 2.1]{QW}.
Albin and Choi \cite{AlbinChoi} contain a novel elementary proof,
together with an indepth discussion of the literature of the subject.

\begin{lemma}[Pickands]\label{lem:Pickands}
	Let $\{\eta(t)\}_{t\ge0}$ denote a continuous stationary Gaussian process
	with $\E\,\eta(t)=0$ and $\Var\eta(t)=1$ for all $t\ge 0$.
	Suppose that there exist constants
	$\kappa\in(0\,,\infty)$ and $\alpha\in(0\,,2]$ such that
	\begin{equation}
		\Cov[\eta(t)\,,\eta(0)] =1- \kappa t^\alpha + o(t^\alpha)
		\qquad\text{as $t\to0^+$.}
	\end{equation}
	Then,
	\begin{equation}
		\P\bigg\{ \sup_{s\in[0,1]}\eta(s)>x\bigg\} =
		\frac{H_\alpha \kappa^{1/\alpha}+o(1)}{(2\pi)^{1/2}}
		x^{(2-\alpha)/\alpha} \e^{-x^2/2}\quad\text{as $x\to\infty$},	
	\end{equation}
	where $H_\alpha\in(0\,,\infty)$ is a numerical constant that depends only on $\alpha$.
\end{lemma}

\begin{remark}
	The cited literature also contains the assumption
	that there exists $h\in(0\,,\infty)$ such that
	$\inf_{t\in[0,h]}\Cov[\eta(t)\,,\eta(0)]>0$. We have omitted it
	as it is subsumed by the assumed behavior of
	$\Cov[\eta(t)\,,\eta(0)]$ near $t=0$.
\end{remark}

\begin{remark}
	The Pickands constant $H_\alpha$ is itself rather interesting. We follow
	Pickands \cite{Pickands} and let
	$\Phi:=\{\Phi(t)\}_{t\ge0}$ denote a centered Gaussian process
	with
	\begin{equation}
		\Cov[\Phi(s)\,,\Phi(t)]=s^\alpha+t^\alpha-|t-s|^\alpha.
	\end{equation}
	The process
	$\Phi$ is a fractional Brownian motion with parameter $\alpha/2$
	when $\alpha\in(0\,,2)$, and $\Phi(t)=t\zeta$ for a standard normal
	random variable $\zeta$ when $\alpha=2$. 
	Then, $H_\alpha$ is the following nontrivial
	limit [which is known to exist as well]:
	\begin{equation}
		H_\alpha 
		= \lim_{T\to\infty} \frac1T\,\E\bigg[ \sup_{t\in[0,T]}
		\e^{\Phi(t) - t^\alpha}\bigg].
	\end{equation}
	It is known that $H_1=1$ and $H_2=\pi^{-1/2}$. Other values of
	$H_\alpha$ are not known. See Harper \cite{Harper} and the references 
	therein for recent estimates.
\end{remark}

\section{Peaks of the Ornstein-Uhlenbeck process}

For a first, and perhaps simplest, example of the general theory 
of \S\ref{sec:Gen} let us continue to write $B$ for a standard Brownian motion on $\R$,
and define
\begin{equation}\label{U}
	U(t) := \e^{-t/2}B(\e^t)\qquad\text{for all $t\ge 0$}.
\end{equation} 
Then $U:=\{U(t)\}_{t\ge 0}$ is an Ornstein--Uhlenbeck process;
that is, $U$ is a centered Gaussian diffusion with
$\E[U(t)U(s)]=\exp(-|t-s|/2)$ for all $s,t\ge 0$.
Thanks to the law of the iterated logarithm for Brownian motion,
\begin{equation}
	\limsup_{t\to\infty}\frac{U(t)}{(2\log t)^{1/2}}=1\qquad\text{a.s.}
\end{equation}
Here we consider the exceedance times of $U$, defined
as follows:
\begin{equation}
	\cP_U(\gamma) := 
	\left\{ t\ge \e:\ \frac{U(t)}{(2\log t)^{1/2}}\ge\gamma\right\}
	\quad(\gamma>0).
\end{equation}
This notation is consistent with the notation in \eqref{LIL:X} and
\eqref{L_X} with $g(x):=(2\log_+ x)^{1/2}$.

Because $\cP_U(\gamma)=\log\cP_B(\gamma)$---%
where $\cP_B(\gamma)$ was defined in \eqref{eq:L_B}---%
and the natural logarithm is strictly monotone, we see that
the random sets $\cP_U(\gamma)$ and $\cP_B(\gamma)$
are bounded and unbounded together. In particular, Proposition
\ref{pr:LIL} implies that $\cP_U(\gamma)$
is unbounded [a.s.] if and only if $\gamma\le 1$; this fact follows
also from the integral test of Motoo \cite{Motoo}. Therefore, the
following theorem computes the macroscopic Hausdorff dimension of 
$\cP_U(\gamma)$ in all cases of interest.

\begin{theorem}\label{th:U}
	Part \eqref{OU:OU} of Theorem \ref{B:OU} holds.
	Namely,
	\begin{equation}
		\Dimh \cP_U(\gamma) = 1- \gamma^2
		\qquad\text{a.s.\ for all $\gamma\in(0\,,1]$}.
	\end{equation}
\end{theorem}

\begin{remark}\label{rem:compare:B:U}
	We can compare Theorems \ref{th:LIL:BM} and
	\ref{th:U} in order to see that 
	$\Dimh\log\cP_B(1)=0$
	a.s.\ whereas $\Dimh\cP_B(1)=1$ a.s. Equivalently,
	\begin{equation}
		\Dimh\exp(\cP_U(1))=1\neq 0=\Dimh\cP_U(1)
		\qquad\text{a.s.}
	\end{equation}
\end{remark}

The elegant theory of Weber \cite[Theorem 6.1]{Weber} implies the following 
closely-related result: With probability one,
\begin{equation}
	\lim_{n\to\infty} \frac1n \log\left|\left\{
	0\le j\le \e^n:\
	\cP_U(\gamma)\cap \left[ j,\,j+1\right)
	\neq\varnothing\right\} \right| = 1-\gamma^2,
\end{equation}
where $|\,\cdots|$ denotes cardinality here. In other words,
\begin{equation}
	\Dimm\cP_U(\gamma) =1-\gamma^2\quad\text{a.s.},
\end{equation}
where $\Dimm$ denotes \emph{macroscopic Minkowski dimension}.
In the notation of Barlow and Taylor \cite{BarlowTaylor,BarlowTaylor1}, 
$\Dimm E$ is the common value of $\dim_{\rm UM}E$
and $\dim_{\rm LM}E$, when the two are equal. Since
$\Dimh E\le\Dimm E$ for all $E\subseteq\R^d$ \cite[Lemma 3.1(i)]{BarlowTaylor1},
Weber's theorem implies half of Theorem \ref{th:U}; i.e.,
$\Dimh\cP_U(\gamma)\le 1-\gamma^2$
a.s. As part of proving the converse inequality, we plan to
use Theorems \ref{th:Gen:LIL:UB} and \ref{th:Gen:LIL:LB} in order to obtain
both inequalities at the same time.

\begin{proof}[Proof of Theorem \ref{th:U}]
	We apply Theorems \ref{th:Gen:LIL:UB} and \ref{th:Gen:LIL:LB}
	with $b=2$, $X_t:=U(t)$ for all $t\ge 0$,
	and $S(x):=x$ for all $x\in\R$. Since $\{U(t)\}_{t\ge 0}$ is stationary,
	we can see from an elementary bound on the tails of the Gaussian law
	that $\bm{c}(2)=\bm{C}(2)=2$. 
	
	Because
	\begin{equation}
		\Cov(U(t)\,,U(0))=\e^{-t/2}=1-\frac{t}{2}+o(t)\qquad\text{as $s\to t$},
	\end{equation}
	Pickands's lemma [Lemma \ref{lem:Pickands}]
	implies the maximal inequality \eqref{cond:UB},
	and our Theorem \ref{th:Gen:LIL:UB} then shows that $\Dimh \cP_U(\gamma)
	\le 1-\gamma^2$ a.s.\ for all $\gamma\in(0\,,1)$. 
	
	In order to prove the reverse
	inequality, let us note that if $t_1<\cdots<t_m$ are in $\cS_n$ and
	$t_{i+1}-t_i\ge\exp\{\delta n\}$ for all $1\le i\le m$, then
	we may set 
	\begin{equation}
		Y_i := \e^{-t_i/2}\left[ B(\e^{t_i})- B(\e^{t_{i-1}})\right]
		\qquad(1\le i\le m),
	\end{equation}
	with $t_0:=\e^n$. It is  easy to see that the $Y_i$'s are independent and
	\begin{equation}
		\max_{1\le i\le m}\E\left( \left| U(t_i) - Y_i\right|^2\right)
		=\e^{-(t_i-t_{i-1})}\le\exp\left\{ -\e^{\delta n}\right\}.
	\end{equation}
	Consequently, a standard bound on the tails of Gaussian laws implies
	that
	\begin{equation}
		\max_{1\le i\le m}\P\left\{ |U(t_i)-Y_i| >1 \right\} \le 2
		\exp\left( -\tfrac12\exp\left\{ \e^{\delta n}\right\}\right).
	\end{equation}
	Of course, this proves that
	\begin{equation}
		\lim_{n\to\infty} \frac1n
		\max_{1\le i\le m}\log\P\{ |U(t_i)-Y_i| >1 \} =-\infty,
	\end{equation}
	with room to spare. Hence, Condition \eqref{cond:LB} is  verified
	since the $Y_i$'s are independent. We can apply
	Theorem \ref{th:Gen:LIL:LB}---with
	$\gamma$ replaced by $\gamma^2$---in order to
	deduce that $\Dimh\cP_U(\gamma)
	\ge1-\gamma^2$ a.s.\ for all $\gamma\in(0\,,1)$. This  completes the proof.
\end{proof}

\section{Peaks of the linear heat equation}\label{sec:linearSHE}

Now we move on to examples that are perhaps more interesting.
Consider the linear stochastic heat equation,
\begin{equation}\label{SHE:Linear:WN}
	\dot{Z}_t(x) = \tfrac12 Z_t''(x) 
	+\xi_t(x)\qquad(x\in\R,\,t>0),
\end{equation}
subject to $Z_0(x):=0$ [say], where $\xi$ denotes space-time white noise.
That, $\xi$ is a totally-scattered centered Gaussian noise with
\begin{equation}
	\Cov(\xi_t(x)\,,\xi_s(y))=\delta_0(s-t)\delta_0(x-y)
	\qquad\text{for $s,t\ge 0$ and $x,y\in\R$.}
\end{equation}
It is well known---see Walsh \cite[Chapter 3]{Walsh} for example---that
there exists a unique integral solution to the stochastic PDE \eqref{SHE:Linear:WN},
and that solution has the following representation in terms of a Wiener
integral:
\begin{equation}
	Z_t(x) = \int_{(0,t)\times\R} p_{t-s}(y-x)\, \xi(\d s\,\d y)\qquad(t>0,\,,x\in\R),
\end{equation}
where the function $(s\,,t\,,x\,,y)\mapsto p_{t-s}(x-y)\1_{(0,\infty)}(t-s)$ denotes the fundamental solution
to the heat operator, 
\begin{equation}
	{\rm L}:=\frac{\partial}{\partial t} - \frac12\frac{\partial^2}{\partial x^2}
	\qquad\text{on $(0\,,\infty)\times\R$}.
\end{equation}
That is,
\begin{equation}\label{p}
	p_t(x) := \frac{\e^{-x^2/(2t)}}{\sqrt{2\pi t}},
\end{equation}
for every $t>0$ and $x\in\R$.

It is also well known---see Walsh [{\it ibid.}]---%
that the random field $Z$ has a modification 
that is continuous in
$(t\,,x)$; from now on we always use that version of the random field $Z$ 
in order to avoid measurability issues.

We are interested in the structure of the tall peaks of the random field
$x\mapsto Z_t(x)$, where $t>0$ is held fixed. With this aim in mind,
consider the random set
\begin{equation}
	\cP_{Z_t}(\gamma) := \left\{ x\ge \exp(\e):\ 
	\frac{Z_t(x)}{(2\log x)^{1/2}}\ge 
	\gamma\left(\frac t\pi\right)^{1/4}\right\},
\end{equation}
where $t,\gamma>0$ are fixed.

\begin{theorem}\label{th:LIL}
	Every $\cP_{Z_t}(\gamma)$ is almost surely unbounded
	if $\gamma\le 1$; else, if $\gamma>1$ then $\cP_{Z_t}(\gamma)$
	is almost surely bounded.
	Furthermore, 
	\begin{equation}
		\Dimh\cP_{Z_t}(\gamma)=1-\gamma^2\qquad\text{a.s.},
	\end{equation}
	for all $t>0$ and $\gamma\in(0\,,1]$.
\end{theorem}

A moment's thought shows that Theorem \ref{th:LIL} is an equivalent reformulation
of Theorem \ref{th:multifractal:he} of Introduction. From a technical
point of view, this particular formulation is more natural to state as well as prove.

As we will see very soon, the fact that $\gamma=1$ is critical for the unboundedness
of $\cP_{Z_t}(\gamma)$ is a fairly routine consequence
of well-known theorems about the growth of a Gaussian process
\cite{QW}. The main
assertion of Theorem \ref{th:LIL} is the one about the
macroscopic Hausdorff dimension of $\cP_{Z_t}(\gamma)$.
Still, let us mention also the following immediate consequence of
the first [more or less routine] portion of Theorem \ref{th:LIL}: 
\begin{equation}\label{LIL:Z}
	\limsup_{x\to\infty} \frac{Z_t(x)}{(2\log x)^{1/2}} = 
	\left(\frac{t}{\pi}\right)^{1/4}\qquad\text{a.s.,}
\end{equation}
for every nonrandom $t>0$. A ``steady state'' version of this fact
appears earlier, for example, in  Collela and Lanford \cite[Theorem 1.1(c)]{CollelaLanford}.
The following lemma puts things in the general framework of Gaussian analysis.

\begin{lemma}\label{lem:linear:Cov}
	Fix some $t>0$. Then,
	$\{Z_t(x)\}_{x\in\R}$ is a stationary Gaussian process
	with $\E Z_t(0)=0$, $\Var Z_t(0)=(t/\pi)^{1/2}$,
	and $\Corr[Z_t(x)\,,Z_t(0)]=O(|x|^{-a})$ as $|x|\to\infty$ for every $a>0$.
	Finally, 
	\begin{equation}
		\Corr[ Z_t(x)\,,Z_t(0)]= 1 - \frac12\left(\frac{\pi}{t}\right)^{1/2}|x| + o(|x|)
		\qquad\text{as $|x|\to 0$.}
	\end{equation}
\end{lemma}

\begin{proof}
	Clearly, $x\mapsto Z_t(x)$ is
	a mean-zero Gaussian process with 
	\begin{equation}
		\Cov\left[ Z_t(x)\,,Z_t(x')\right] = \int_0^t p_{2s}(x-x')\,\d s
		\quad\text{for all $x,x'\in\R$}.
	\end{equation}
	It follows from this formula that $Z_t(\bullet)$ is stationary as well,
	and has variance
	\begin{equation}
		\Var Z_t(0) = \int_0^t p_{2s}(0)\,\d s=\int_0^t(4\pi s)^{-1/2}\,\d s
		=\left(\frac t\pi\right)^{1/2}.
	\end{equation}
	Furthermore, the preceding display shows also that
	\begin{equation}
		\Cov[ Z_t(x)\,, Z_t(0)] = \int_0^t p_{2s}(x)\,\d s
	\end{equation}
	is bounded above by a finite constant $C(t)$ times $\e^{-x^2/(4t)}$,
	and hence goes to zero faster than any negative
	power of $|x|$, as $|x|\to\infty$. Finally, we note that
	if $x>0$, then
	\begin{align}\label{eq:for:inf:claim}
		\Var(Z_t(0))-\Cov\left[ Z_t(x)\,, Z_t(0)\right] &=\int_0^t
			\left[ p_{2s}(0)-p_{2s}(x)\right]\d s\\\notag
		&= \frac{x}{4\sqrt{\pi}}\int_0^{4t/x^2}r^{-1/2}
			\left(1-\e^{-1/r}\right)\d r\\\notag
		&=\frac{x}{4\sqrt{\pi}}\int_0^\infty 
			r^{-1/2} \left(1-\e^{-1/r}\right)\d r
			+ O(x^2),
	\end{align}
	as $x\downarrow 0$. It follows readily from this and symmetry that
	\begin{equation}
		\Corr[Z_t(x)\,,Z_t(0)] =1-c|x|+O(x^2)\qquad
		\text{as $x\to\ 0$,}
	\end{equation}
	with 
	\begin{equation}
		c := \frac{1}{4\sqrt{t}}\int_0^\infty r^{-1/2}
		\left(1-\e^{-1/r}\right)\d r.
	\end{equation}
	A change of variables shows that
	\begin{equation}
		\int_0^\infty r^{-1/2}(1-\e^{-1/r}) \,\d r=
		\int_0^\infty s^{-3/2}(1-\e^{-s}) \,\d s.
	\end{equation}
	Write
	$1-\e^{-s} =\int_0^s \e^{-y}\,\d y$ and apply the Tonelli theorem in order to see that
	$c= \sqrt{\pi/4t}.$ 
\end{proof}

We are ready to establish Theorem \ref{th:LIL}.

\begin{proof}[Proof of Theorem \ref{th:LIL}]
	Throughout the proof, we hold fixed an arbitrary  $t>0$.
	
	Lemma \ref{lem:linear:Cov} verifies 
	all of the conditions of Theorem 1.1 of Qualls and Watanabe \cite{QW},
	and hence it follows from that result that $\cP_{Z_t}(\gamma)$
	is a.s.\ bounded  if $\gamma>1$ and a.s.\ unbounded if
	$\gamma\le 1$. In particular, we obtain \eqref{LIL:Z} immediately. 
	Furthermore, we can see---using the notation of \S\ref{sec:Gen}---that
	\begin{equation}
		\text{$b=2$ and 
		$\bm{c}(b)=\bm{C}(b)=\left(\frac{\pi}{4t}\right)^{1/2}$.}
	\end{equation}
	
	Thanks to Lemma \ref{lem:linear:Cov}, the assumptions of Lemma
	2.5 of Pickands \cite{Pickands} are met. Lemma \ref{lem:Pickands}
	[Pickands' theorem] implies 
	the maximal inequality \eqref{cond:UB}; therefore, we may apply
	Theorem \ref{th:Gen:LIL:UB}---with $\gamma$ replaced by
	$\gamma^2$---in order to conclude that
	\begin{equation}\label{DimZ:LB}
		\Dimh\cP_{Z_t}(\gamma)\le1-\gamma^2,
	\end{equation}
	a.s.\ for all $\gamma\in(0\,,1]$.
	We plan to prove a matching lower bound by appealing to 
	Theorem \ref{th:Gen:LIL:LB} with $S(x):=x$ for all $x\in\R$. 
	Therefore, it remains to verify the coupling 
	assumption \eqref{cond:LB}, which we do next.
	
	For every $B>0$ we may define a
	space-time Gaussian random field $Z^{(B)}$ as follows: For all $x\in \R$,
	\begin{equation}\label{eq:ZB}
		Z^{(B)}_t(x) := \int_{(0,t)\times[x-(Bt)^{1/2},\ x+(Bt)^{1/2}]} 
		p_{t-s}(y-x)\,\xi(\d s\,\d y).
	\end{equation}
	It is intuitively clear that that $Z\approx Z^{(B)}$ when $B\gg1$.
	Next we  claim the following quantitative improvement of this remark:
	For all $t,B,\lambda>0$,
	\begin{equation}\label{claim:Linear}
		\sup_{x\in\R}
		\P\left\{ \left| Z_t(x) - Z^{(B)}_t(x)\right|>\lambda\right\}
		\le 2\exp\left( -\frac{\lambda^2}{2}\sqrt{\frac{\pi}{8t}}\,
		\e^{B/2}\right).
	\end{equation}
	Indeed, because $p_s(z)\le p_s(0)= (2\pi s)^{-1/2}$ for all $s>0$ and $z\in\R$,
	the Wiener isometry yields
	\begin{equation}\begin{split}
		\E\left(\left| Z_t(x) - Z^{(B)}_t(x)\right|^2\right) &=
			\int_0^t\d s\int_{|z|>(Bt)^{1/2}} \d z\ [p_s(z)]^2\\
		&\le \int_0^t\frac{\d s}{\sqrt{2\pi s}} \P\left\{ |X|>(Bt/s)^{1/2}\right\},
	\end{split}\end{equation}
	where $X$ has a standard normal distribution. If $s\in(0\,,t)$, then 
	we combine the elementary bound,
	\begin{equation}
		\P\left\{|X|>\left(\frac{Bt}{s}\right)^{1/2}\right\}
		\le\P\left\{|X|>\sqrt B\right\},
	\end{equation}
	with a standard bound 
	on the tails of the standard normal distribution in order to see
	\begin{equation}\label{eq:VarZ-ZB}
		\Var\left(Z_t(x)-Z^{(B)}_t(x)\right)\le
		\left(\frac{8t}{\pi}\right)^{1/2}\e^{-B/2}.
	\end{equation}
	The claim \eqref{claim:Linear}  
	follows readily from this and another appeal to the tails
	of the Gaussian laws. 
	
	We use \eqref{claim:Linear} in order to prove
	\eqref{cond:LB} using the following.
	\begin{obs}\label{obs1}
		If $x_1,x_2,\ldots,x_m\in\R$ satisfy $|x_i-x_j|>2(Bt)^{1/2}$
		when $1\le i\ne j\le m$, then the random variables $Z^{(B)}_t(x_1),\ldots,Z^{(B)}_t(x_m)$ 
		are independent.
	\end{obs}
	
	Choose and fix some $\delta\in(0\,,1)$.
	If $\e^n\le x_1<\cdots<x_m<\e^{n+1}$ are $m$ arbitrary points in $\cS_n$ 
	such that $x_{i+1}-x_i\ge\exp\{\delta n\}$, then we set
	$Y_j:=Z^{(n)}_t(x_j)$ for all $1\le j\le m$. Thanks to Observation \ref{obs1},
	$Y_1,\ldots,Y_m$ are independent random variables
	as long as $n$ is large enough to ensure
	that $2\sqrt{nt}<\exp\{\delta n\}$. And \eqref{claim:Linear}
	ensures that
	\begin{equation}
		\max_{1\le i\le m}\P\left\{ \left| Z_t(x_i) - Y_i\right|> 1\right\}
		\le 2\exp\left(-\frac12\sqrt{\frac{\pi}{8t}}\, \e^{n/2}\right).
	\end{equation} 
	In particular, 
	\begin{equation}
		\lim_{n\to\infty} \frac1n \max_{1\le i\le m}\log
		\P\{ | Z_t(x_i) - Y_i|> 1\}=-\infty.
	\end{equation}
	This implies \eqref{cond:LB} readily, and the lower bound
	that complements \eqref{DimZ:LB} follows
	from the conclusion of Theorem \ref{th:Gen:LIL:LB}.
\end{proof}

\section{Peaks of a nonlinear stochastic heat equation}\label{sec:she}

Let us now consider the following nonlinear stochastic partial differential equation,
\begin{align}\label{eq:pam}
	\dot{u}_t(x) &= 
		\tfrac12u_t''(x) +\sigma(u_t(x)) \xi_t(x)\qquad(x\in\R,\, t>0),\\
	u_0(x)&=1,
\end{align}
where $\xi$ denotes space-time white noise, as before, and
$\sigma: \R\rightarrow \R$ is a Lipschitz continuous 
and non-random function with $\sigma(0)=0$.

It is well-known that the stochastic heat equation \eqref{eq:pam} 
has a unique solution; see Dalang \cite{Dalang}, for instance. And 
that solution solves the following stochastic integral equation,
interpretted in the sense of Walsh \cite{Walsh}:
\begin{equation}
	u_t(x)=1+\int_{(0,t)\times\R} p_{t-s}(y-x) \sigma(u_s(y)) \xi(\d s\, \d y),
\end{equation}
where $p_t(x)$ is the standard heat kernel on $(0\,,\infty)\times\R$; 
see \eqref{p}. 

It is known also that the solution to \eqref{eq:pam}
is strictly positive for all $t>0$; see Mueller \cite{Mueller}
for a closely-related statement.  The precise positivity assertion 
that is required here follows from the work of 
Mueller and Nualart \cite{MuellerNualart}. 
Therefore, the tall peaks of $x\mapsto u_t(x)$ 
and $x\mapsto h_t(x)$ match, where
\begin{equation}\label{h=log(u)}
	h_t(x) := \log u_t(x).
\end{equation}

The random field $h$ is particularly well studied when $\sigma(z)= z$
for all $z\in\R$. In that case, a formal change of variables suggests that
\begin{equation}\label{eq:KPZ}
	\dot{h}_t(x) = \tfrac12 h''_t(x) + \tfrac12(h'_t(x))^2 - \xi_t(x),
\end{equation}
subject to $h_0(x)\equiv 0$. This purely-formal ``computation'' is
analogous to the classical Hopf--Cole solution to Burgers' equation,
and is due to Kardar et al \cite{KPZ}. The resulting ill-posed stochastic PDE
\eqref{eq:KPZ} is the so-called ``KPZ equation'' of statistical mechanics.
The recent solution theory of Hairer \cite{Hairer} gives a meaning to the analogous
version of \eqref{eq:KPZ} where the $x$ variable lives in $[0\,,1]$
[together with suitable boundary conditions]. As far as we know, the original
problem on $\R$ has not yet been given a rigorous meaning.

In this section, we plan to study the set of points $x>\exp(\e)$ 
at which the solution $u_t(x)$ exceeds certain high peaks. 
For the parabolic Anderson model---that is when 
$\sigma(z)=z$ for all $z\in\R$---Conus et al \cite{CJK} have demonstrated 
that, for every $t>0$, the tall peaks of $x\mapsto h_t(x)$ are of rough height 
$(\log|x|)^{2/3}$ as $|x|\to\infty$. Specifically, they have proved that
\begin{equation}\label{NL:heat:LIL}
	0<\limsup_{x\to\infty} \frac{h_t(x)}{(\log x)^{2/3}}<\infty\qquad\text{a.s.},
\end{equation}
for all $t>0$. 
Conus et al [{\it ibid.}] have also proved that the function
$(\log x)^{2/3}$ fails to correctly gauge the height of the tall peaks of $h_t(x)$
for general nonlinearities $\sigma$. 

We will prove among other things
that \eqref{NL:heat:LIL} holds when $|\sigma(z)/z|$ is bounded uniformly from below by 
a positive constant. The mentioned boundedness condition is known to
be an \emph{intermittency condition} for the system \eqref{eq:pam} \cite{FA:whitenoise}.

In order to describe our results in greater details let us define
\begin{equation}\label{ell}
	\ell_\sigma:=\inf_{z\in\R\setminus\{0\}} | \sigma(z)/z|,\quad
	L_\sigma:=\sup_{z\in\R\setminus\{0\}}|\sigma(z)/z|.
\end{equation}
Because $\sigma$ is Lipschitz continuous we always have
$0\le \ell_\sigma\le L_\sigma\le\infty$. We will be assuming that
\begin{equation}\label{assump:lip}
	0<\ell_\sigma\leq L_\sigma <\infty.
\end{equation}
We call \eqref{assump:lip} an ``intermittency condition'' because
it is the only known condition under which the solution to \eqref{eq:pam}
is known to be intermittent in the sense that
\begin{equation}
	k\mapsto\frac{\lambda(k)}{k}\text{ is strictly increasing on
	$[2\,,\infty)$,}
\end{equation}
where either
\begin{equation}
	\lambda(k) := \limsup_{t\to\infty} t^{-1}\sup_{x\in\R}
	\E\left(|u_t(x)|^k\right),
\end{equation}
or
\begin{equation}
	\lambda(k) := \liminf_{t\to\infty} t^{-1}\inf_{x\in\R}
	\E\left(|u_t(x)|^k\right),
\end{equation}
describe respectively the top and bottom $k$th moment Lyapunov
exponents of the solution;
see  Foondun and Khoshnevisan \cite[Theorem 2.7]{FA:whitenoise}.\footnote{In
general, the Lyapunov exponents, as were describe in the introduction, do
not exists.}

Now define for all $t,\gamma>0$,
\begin{equation}\label{L_h:WN}
	\cP_{h_t}(\gamma) :=
	\left\{ x\ge \exp(\e): \frac{h_t(x)}{(\log x)^{2/3}}\ge \gamma t^{1/3}\right\}.
\end{equation}
We will  use the general theory of
\S\ref{sec:Gen} in order to prove the following, which is the main result of 
this section. It might help to recall yet again our earlier convention that when
we state that $\Dimh E<0$ we mean that
$E$ is bounded.

\begin{theorem}\label{th:pam:LIL}
	Under \eqref{assump:lip}, the following holds with probability one:
	\begin{equation}\label{eq:pam:LIL}
		\left(\frac{9}{32}\right)^{1/3}\ell_\sigma^{4/3}\le
		\limsup_{x\to\infty} \frac{h_t(x)}{t^{1/3}(\log x)^{2/3}}\le
		\left(\frac{9}{32}\right)^{1/3}L_\sigma^{4/3};
	\end{equation}
	for every $t,\gamma>0$. Moreover,
	\begin{equation}
		1-\bm{\alpha}\gamma^{3/2}\le\Dimh\cP_{h_t}(\gamma)
		\le 1-\bm{\beta}\gamma^{3/2}
		\qquad\text{a.s.},
	\end{equation}
	where $\cP_{h_t}(\gamma)$ was defined in \eqref{L_h:WN}, 
	\begin{equation}\label{alpha:beta}
		\bm{\alpha}:=\frac{4\sqrt{2}}{3\ell^2_\sigma},
		\quad\text{and}\quad
		\bm{\beta}:=\frac{4\sqrt{2}}{3L^2_\sigma}.
	\end{equation}
\end{theorem}

When $\sigma(z)=z$ for all $z\in\R$,
the stochastic PDE \eqref{eq:pam} simplifies to the
following, which is known as a \emph{parabolic Anderson model}
and/or \emph{diffusion in random white-noise potential}:
\begin{equation}\label{eq:pampam}\left[\begin{split}
	\dot{u}_t(x) &= \tfrac12 u''_t(x) + u_t(x)\xi_t(x)\qquad(x\in\R,\, t>0);\\
	u_0(x) &= 1.
\end{split}\right.\end{equation}
In this case, Theorem \ref{th:pam:LIL} yields the following exact formula,
which is an equivalent but perhaps more explicit 
way to state Theorem \ref{th:multifractal:pam}.

\begin{corollary}\label{co:PAM:WN}
	The solution $u$ to \eqref{eq:pampam} satisfies
	the following: For all $\gamma,t>0$, 
	\begin{equation}
		\Dimh\left\{ x\ge \exp(\e): u_t(x)\ge 
		\e^{\gamma t^{1/3}(\log x)^{2/3}}\right\}
		=1-\frac{4\sqrt{2}}{3}\,\gamma^{3/2}
		\qquad\text{a.s.}
	\end{equation}
\end{corollary}

Let us mention a rather general corollary  of Theorem \ref{th:pam:LIL} as
well.

\begin{corollary}\label{co:LIL:pam}
	Under \eqref{assump:lip}, the tall peaks of the
	solution to the SPDE \eqref{eq:pam}
	are almost surely multifractal.
\end{corollary}

\begin{proof} 
	Recall \eqref{alpha:beta},
	let $\gamma_1:=(2\bm{\alpha})^{-2/3}$, and then define
	\begin{equation}
		\gamma_{i+1} := \left( \frac{\bm{\beta}}{2\bm{\alpha}}\right)^{2/3}
		\gamma_i,
	\end{equation}
	iteratively for all $i\ge 1$. Clearly, 
	\begin{equation}
		1-\bm{\beta}\gamma_i^{3/2}<
		 1 -\frac12\bm{\beta}\gamma_{i}^{3/2} = 1 -\bm{\alpha}\gamma_{i+1}^{3/2}
		\qquad\text{for all $i\ge 1$}.
	\end{equation}
	In addition, $0< 1 -\bm{\alpha}\gamma_{i}^{3/2} <1-\bm{\beta}\gamma_i^{3/2}<1$ for all $i\ge 1$.

	Consider the collection of tall peaks $\cP_{h_t}(\gamma)$ 
	of order $\gamma\in(0\,,1)$ that was defined in 
	\eqref{L_h:WN}.  Theorem \ref{th:pam:LIL} implies that
	\begin{equation}
		\Dimh \cP_{h_t}(\gamma_{i}) < \Dimh\cP_{h_t}(\gamma_{i+1})
		\qquad\text{for all $i\ge 1$, a.s.}
	\end{equation}
	Definition \ref{def:multifractal} then shows that the tall peaks of
	$h_t$---whence also $u_t$---are a.s.\ multifractal.
\end{proof}

We begin the proof of Theorem \ref{th:pam:LIL}
with a basic tail probability estimate.

\begin{proposition}\label{pr:pam:tail}
	For any $t>0$, we have
	\begin{equation}\label{eq:pam:tail}\begin{split}
		& \liminf_{z\rightarrow \infty}  z^{-3/2}\inf_{x\in\R}
			\log \P\left\{ h_t(x) \geq z\right\}\geq - \frac{\bm{\alpha}}{\sqrt{t}},\\
		& \limsup_{z\rightarrow \infty}  z^{-3/2}\sup_{x\in\R}
			\log \P\left\{ h_t(x) \geq z\right\}\leq - \frac{\bm{\beta}}{\sqrt{t}},
	\end{split}\end{equation} 
	where the constants $\bm{\alpha}$ and $\bm{\beta}$ 
	were defined in \eqref{alpha:beta}.
\end{proposition} 

\begin{proof}
	Let $u^{(\ell)}$ and $u^{(L)}$ respectively denote the solutions to 
	\eqref{eq:pam} with $\sigma(z):=\ell_\sigma z$ and $\sigma(z):=L_{\sigma} z$. 
	The \emph{moment comparison principle} of Joseph et al 
	\cite[Theorem 2.6]{JKM} tells us that because of the condition \eqref{assump:lip},
	\begin{equation}
		\E \left(\left[u^{(\ell)}_t(x)\right]^k\right) \le 
		\E \left([ u_t(x)]^k\right) \le
		\E \left(\left[u^{(L)}_t(x)\right]^k\right),
	\end{equation}
	for all real numbers $t>0$, $x\in\R$, and $k\ge 2$.
	We now use the first part of Theorem 5.5 of Chen \cite{Chen} and 
	the G\"{a}rtner--Ellis theorem, for example in the form of Chapter 1 of the recent book by
	Chen \cite{Chen-book}, in order to complete the proof. 
\end{proof}

Armed with Proposition \ref{pr:pam:tail} we can prove half of Theorem \ref{th:pam:LIL}
quickly. The second, harder, half will require work that will be developed
afterward.

\begin{proof}[Proof of Theorem \ref{th:pam:LIL}: Dimension upper bound]
	Our immediate goal is to establish the dimension upper bound; that is, we wish to
	demonstrate the following:
	\begin{equation}\label{BOOBOO}
		\Dimh \cP_{h_t}\left(
		[\gamma/\bm{\beta}]^{2/3}\right)\le 1-\gamma
		\quad\text{a.s.\ for all $\gamma>0$}.
	\end{equation}
	
	We claim that for every $\gamma\in(0\,,1)$,
	\begin{equation}\label{eq:pam:h}
		\sup_{z\in\R}\P\left\{ \sup_{y\in[z,z+1]} h_t(y)
		> \left(\frac{\gamma}{\bm{\beta}}\log s\right)^{2/3}
		\right\} \le s^{-\gamma+o(1)}\quad\text{as $s\to\infty$}.
	\end{equation}
	If this were so, then it would show that Condition \eqref{cond:UB}
	holds with $b=\nicefrac32$
	and $X_x:= h_t(x)$, and  Theorem \ref{th:Gen:LIL:UB} then implies
	\eqref{BOOBOO}. Thus, the dimension upper bound of Theorem \ref{th:pam:LIL}
	follows once we prove \eqref{eq:pam:h}.
	According to \eqref{h=log(u)}, it remains to prove that
	for every $\gamma\in(0\,,1)$,
	\begin{equation}\label{eq:pam:U}
		\sup_{z\in\R}\P\left\{ \sup_{y\in[z,z+1]} u_t(y)
		> \exp\left[\left( \frac{\gamma}{\bm{\beta}}\log s\right)^{2/3}
		\right]\right\} \le s^{-\gamma+o(1)},
	\end{equation}
	as $s\to\infty$.
	
	Recall that, as a corollary to Proposition \ref{pr:pam:tail}, we have  
	the following slightly weaker variation on the desired estimate \eqref{eq:pam:U}:
	If $g:\R_+\to\R_+$ is a nonrandom function
	that satisfies $\lim_{s\to\infty}g(s)=0$, then
	\begin{equation}\label{eq:pam:U0}
		\sup_{y\in\R}\P\left\{ u_t(y)
		> \exp\left[\left( \frac{\gamma
		-g(s)}{\bm{\beta}}\log s\right)^{2/3}
		\right]\right\} \le s^{-\gamma+o(1)},
	\end{equation} as $s\to\infty$.
		
	In order to derive \eqref{eq:pam:U} from \eqref{eq:pam:U0} we
	apply a chaining argument. With this in mind, let
	us first observe the following, which is a quantitative form of the
	Kolmogorov continuity theorem \cite[Theorem C.6, p.\ 107]{cbms}: 
	There exists a finite constant $\tau=\tau(t)>1$ such that 
	for all real numbers $k\ge 2$
	\begin{equation}
		\sup_{w\in\R} \E\left[
		\sup_{\substack{x,x'\in [w,w+1]\\
		x\neq x'}} \frac{|u_t(x)-u_t(x')|^{2k}}{|x-x'|^{k/2}}
		\right]< \tau \e^{\tau k^3 t}.
	\end{equation}
	This and the Chebyshev inequality together imply that
	uniformly for all real numbers $\eta,\varepsilon\in(0\,,1)$,
	$s\ge\e$, and $k\ge 2$,
	\begin{equation}\label{eq:BD1}\begin{split}
		&\sup_{w\in\R}\P\left\{\sup_{y\in[w,w+\varepsilon]}|u_t(w)-u_t(y)|>
			\exp\left[(\eta\log s)^{2/3}\right]\right\}\\
		&\hskip2in\le \tau \varepsilon^{k/2}
			\exp\left(\tau k^3 t -2 k(\eta\log s)^{2/3}\right).
	\end{split}\end{equation}
	We apply the preceding  bound with the following choices of
	$\varepsilon$ and $k$:
	\begin{equation}\begin{split}
		\varepsilon &= \varepsilon(s) :=\exp\left\{
			- \frac{2\gamma(2\tau t)^{1/2}}{\eta^{1/3}}(\log s)^{2/3}\right\},\\
		k&=k(s) := \left(\frac{2}{\tau t}\right)^{1/2}(\eta\log s)^{1/3}.
	\end{split}\end{equation}
	Let $s_*(\eta\,,t):= \exp\left( \eta^{-1}[2\tau t]^{3/2}\right)$ to see that $k\ge 2$
	if and only if $s\ge s_*(\eta\,,t)$. We apply \eqref{eq:BD1} with these
	choices of $\varepsilon$ and $k$ in order to see that
	\begin{equation}\label{eq:BD2}
		\sup_{w\in\R}\P\left\{\sup_{y\in[w,w+\varepsilon]}|u_t(w)-u_t(y)|>
		\exp\left[(\eta\log s)^{2/3}\right]\right\}
		\le \tau s^{-2\gamma},
	\end{equation}
	uniformly for all $\eta\in(0\,,1)$ and $s\ge \max\{\e\,,s_*(\eta\,,t)\}$.
	In particular, we might note that
	\begin{align}\notag
		&\sup_{w\in\R}\P\left\{ \sup_{y\in[w,w+\varepsilon]} u_t(y)
			> \exp\left[\left( \frac{\gamma}{\bm{\beta}}\log s\right)^{2/3}
			\right]\right\} \\
			\notag
		&\le \tau s^{-2\gamma} + \sup_{w\in\R}
			\P\left\{ u_t(w)
			> \exp\left[\left( \frac{\gamma}{\bm{\beta}}\log s\right)^{2/3}
			\right]-\exp\left[ (\eta\log s)^{2/3}\right]\right\} \\
		&\le s^{-\gamma+o(1)}\qquad\text{as $s\to\infty$},
			\end{align}
	owing to \eqref{eq:pam:U0}. Every interval $[z\,,z+1]$ can be covered by
	at most $\varepsilon^{-1}+1$ intervals of the form
	$[w\,,w+\varepsilon]$. Therefore, the preceding implies that
	\[
		\sup_{z\in\R}\P\left\{ \sup_{y\in[z,z+1]} u_t(y)
		> \exp\left[\left( \frac{\gamma}{\bm{\beta}}\log s\right)^{2/3}
		\right]\right\} 
		\le \left[ \varepsilon^{-1}+1\right] s^{-\gamma+o(1)}
		=s^{-\gamma+o(1)},
	\]
	as $s\to\infty$.
	This proves \eqref{eq:pam:U}, and hence the upper bound on the
	macroscopic Hausdorff dimension in Theorem \ref{th:pam:LIL}.
\end{proof}

Now we begin to work toward establishing the lower bound 
in Theorem \ref{th:pam:LIL}. In order to do that we will attempt
to verify the coupling condition \eqref{cond:LB}. A first attempt might be to follow
the case of linear SPDEs/Gaussian processes. More concretely, we may
try to follow the proof of the lower bound of Theorem \ref{th:LIL} and consider,
for every $B>0$, a space-time random field $u^{(B)}$ as follows: For all $t>0$ and $x\in \R$,
\begin{equation}\label{eq:uB}\begin{split}
	&u^{(B)}_t(x)\\
	&:=1+ \int_{(0,t)\times[x-(Bt)^{1/2},\ x+(Bt)^{1/2}]} 
		p_{t-s}(y-x)\sigma\left(u_s^{(B)}(y)\right)\,\xi(\d s\,\d y).
\end{split}\end{equation} 
It is not hard to apply a fixed-point agrument in order
to prove that the random integral equation \eqref{eq:uB}
has a unique solution $u^{(B)}$. 
Moreover, that $(t\,,x)\mapsto u^{(B)}_t(x)$ has a continuous modification.
[We will not prove any of this here since we will not need to.]

The random field $u^{(B)}$ is the analogue of the random field $Z^{(B)}$,
that was defined earlier in \eqref{eq:ZB},
but we now interpret the stochastic integral in \eqref{eq:uB}
in the sense of Walsh, whereas the one for $Z^{(B)}$
can be understood in the sense of Wiener. 

The random field $Z^{(B)}$ was introduced in the proof of Theorem \ref{th:LIL}
because $Z^{(B)}$ has the following two desireable properties: 
\begin{enumerate}
	\item[(i)] $Z^{(B)}\approx Z$
		if $B$ is large [see \eqref{claim:Linear}]; and 
	\item[(ii)]  $Z^{(B)}_t(x_1),\ldots,Z^{(B)}_t(x_m)$ are independent
		if the $x_i$'s are sufficiently far apart from one another; 
		for example, if $|x_i-x_{i+1}|>2(Bt)^{1/2}$ [see Observation \ref{obs1}.]
\end{enumerate}

By analogy, we might hope that: 
\begin{enumerate}
	\item[(iii)] $u^{(B)}\approx u$ if $B$ is large; and
	\item[(iv)] $u^{(B)}_t(x_1),\ldots,u^{(B)}_t(x_m)$ are independent if the $x_i$'s are
		sufficiently far apart from one another. 
\end{enumerate}
If so, then we could use $u^{(B)}$---in
a similar way as we used $Z^{(B)}$---in order to verify Condition \eqref{cond:LB},
thereby obtain a lower bound on the macroscopic dimension of the set of high peaks
of $u$. 

As it turns out, (iii) continues to hold. However, (iv) is manifestly false; it is possible
for example to show that the covariance of $u^{(B)}_t(x_1)$ and $u^{(B)}_t(x_2)$
is strictly positive, for all $x_1,x_2\in\R$, no matter how far apart $x_1$
and $x_2$ are from one another.

We remedy the situation by defining the following random fields instead:
Choose and fix an integer $B\ge 1$, as before, and define 
\begin{equation}
	u^{(B,0)}_t(x) :=1\qquad\text{for all $t\ge 0$ and $x\in\R$}.
\end{equation}
Then we define random fields $u^{(B,j)}$, for every $j\ge 1$, iteratively, as follows:
\[
	u^{(B,m)}_t(x) 
	:=1+ \int_{(0,t)\times[x-(Bt)^{1/2},\ x+(Bt)^{1/2}]} 
	p_{t-s}(y-x)\sigma\left(u_s^{(B,m-1)}(y)\right)\xi(\d s\,\d y),
\]
for every $m\ge 1$. 
The object of interest to us is the random field $u^{(B,B)}$.
The following estimate  shows that $u\approx u^{(B,B)}$ when $B\gg1$. 

\begin{lemma}[A coupling lemma]\label{lem:u-uBm}
	There exists a finite constant $K>1$ such that  
	for all real numbers $t>0$ and $\lambda>1$, and 
	all integers $B\ge 1$, 
	\begin{equation}
		\sup_{x\in\R}
		\P\left\{ \left| u_t(x) - u^{(B,B)}_t(x)\right|>\lambda\right\}
		\le K\exp\left( -\frac{(B+\log \lambda)^{3/2}}{K\sqrt t}\right).
	\end{equation}
\end{lemma}

\begin{proof}
	According to Lemma 4.3 of 
	Conus et al \cite{CJK}---see also
	Lemma 10.10 of Khoshnevisan \cite{cbms}---%
	there exists a finite constant $c$ such that
	\begin{equation}
		\sup_{x\in\R}
		\E\left(\left|u_t(x)-u_t^{(B,B)}(x)\right|^k\right) 
		\leq  c \exp\left( ck^3 t-Bk\right),
	\end{equation}
	uniformly for all real numbers $k\ge 2$, $\lambda>1$,
	and $t>0$. Therefore, Chebyshev's inequality shows that
	\begin{equation}
		\P\left\{ \left|u_t(x)-u_t^{(B,B)}(x)\right|>\lambda \right\}
		\le c\inf_{k\ge 2}\exp\left( ck^3 t-
		(\log\lambda +B)k\right),
	\end{equation}
	uniformly for all real numbers $t>0$ and $x\in\R$, and 
	for all integers $B\ge 1$.
	This readily implies the lemma.
\end{proof}

Our next lemma is essentially due to Conus et al 
\cite{CJK}, and shows that the random process $x\mapsto u^{(B,B)}_t(x)$
decouples fairly rapidly.

\begin{lemma}\label{lem:pam:ind}
	Suppose $t>0$ is a real number,
	$B\ge 1$ is an integer, and $x_1,\ldots,x_m$
	are points in $\R$ such that $|x_i-x_j|>2B^{3/2}\sqrt{t}$
	whenever $i\neq j$. Then, $u^{(B,B)}_t(x_1),\ldots,u^{(B,B)}_t(x_m)$
	are independent random variables.
\end{lemma}

\begin{proof}
	We include a proof for the sake of completeness.
	
	First, let us observe that if $|x_i-x_j|>2(Bt)^{1/2}$,
	whenever $i\neq j$,
	then $u^{(B,1)}_t(x_1),\ldots,u^{(B,1)}_t(x_m)$ 
	are independent. This is because: (i)
	\begin{equation}
		u^{(B,1)}_t(x) = 1 + \int_{(0,t)\times [ x-(Bt)^{1/2}\,,x+(Bt)^{1/2}]}
		p_{t-s}(y-x)\sigma(1)\,\xi(\d s\,\d y);
	\end{equation}
	and (ii) If $\psi_1,\ldots,\psi_m\in L^2(\R_+\times\R)$ are nonrandom with
	disjoint support then the Wiener integrals $\int\psi_j\,\d\xi$ $(1\le j\le m)$
	are independent [compute covariances]. Next we apply induction,
	using the following induction hypothesis: Suppose that whenever $|x_i-x_j|>
	2\ell(Bt)^{1/2}$ for $i\neq j$, 
	the random variables $u^{(B,\ell)}_t(x_1),\ldots,u^{(B,\ell)}_t(x_m)$
	are independent. Then we wish to prove that if $|x_i-x_j|>
	2(\ell+1)(Bt)^{1/2}$ [$i\neq j$], then 
	$u^{(B,\ell+1)}_t(x_1),\ldots,u^{(B,\ell+1)}_t(x_m)$
	are independent.  This property
	follows readily from the properties of the Walsh stochastic
	integral; namely, that if $\Phi^1,\ldots,\Phi^m$ are independent
	predictable random fields and $\psi^1,\ldots,\psi^m\in L^2(\R_+\times\R)$
	are nonrandom with disjoint support, then the Walsh integrals
	$\int\psi_j\Phi^j\,\d\xi$ are independent $[1\le j\le m]$. 
	This completes our induction argument, and prove
	the lemma. 
\end{proof}

 We can now verify the dimension lower bound of Theorem \ref{th:pam:LIL}.
 
\begin{proof}[Proof of Theorem \ref{th:pam:LIL}: Dimension lower bound]
	Our proof of Theorem \ref{th:pam:LIL} will be complete once we demonstrate
	that 
	\begin{equation}\label{BOOBOOB}
		\Dimh\cP_{h_t}\left(
		[\gamma/\bm{\alpha}]^{2/3}\right)\ge 1-\gamma
		\quad\text{a.s.\ for all $\gamma>0$}.
	\end{equation}
	In order to establish this fact we will appeal to
	Theorem \ref{th:Gen:LIL:LB}; therefore, it remains to verify Condition \eqref{cond:LB}.
	We will appeal, as we did in the proof of the upper bound, to the general
	theory of Section \ref{sec:Gen}, using the identifications $X_x:=h_t(x)$,
	$b:=3/2$, and $S(x):=\exp(x)$.
	
	Let us choose and fix an integer $n\ge 1$ and a real number $\delta\in(0\,,1)$,
	and consider an arbitrary collection $\{x_i\}_{i=1}^m$ of points such 
	that:
	(a) $\e^n\le x_1<\cdots<x_m<\e^{n+1}$; and (b)
	$x_{i+1}-x_i\ge\exp(\delta n)$. From now on, we set
	$B := n^3$ and 
	\begin{equation}
		Y_j :=  \log u^{(B,B)}_t(x_j)\qquad(1\le j\le m).
	\end{equation}
	According to Lemma \ref{lem:pam:ind}, $Y_1,\cdots,Y_m$ are independent
	as long as $2n^3t<\exp(\delta n)$; and Lemma
	\ref{lem:u-uBm} ensures that
	\begin{equation}
		\max_{1\le j\le m}
		\P\left\{ |S(X_{x_j}) - S(Y_j) | >1 \right\} \le K\exp\left( -
		\frac{n^{3/2}}{K\sqrt t}\right).
	\end{equation}
	[Recall that $S(x):=x$ here.]
	Since the constant $K$ does not depend on the choice of $x_1,\ldots,x_m\in\cS_n$,
	we have shown that Condition \eqref{cond:LB} holds. Theorem
	\ref{th:Gen:LIL:LB} implies \eqref{BOOBOOB}, and completes our
	proof of Theorem \ref{th:pam:LIL}.
\end{proof}

\section{$d$-dimensional diffusion in random potential}\label{sec:pamcolor}

We will conclude this paper by presenting examples of stochastic PDEs
over $\R_+\times\R^d$ where $d\ge 1$ is an arbitrary positive integer. 
In order to keep the ensuing theory at a reasonably-modest technical level, 
we focus only on linear stochastic partial differential equations of
the following type:
\begin{equation}\left[\label{E:SHE:color}\begin{split}
	\dot{u}_t(x)&=\tfrac12(\Delta u_t)(x)+ u_t(x)\eta_t(x) 
		\quad \text{for $t>0, x\in\R^d$};\\
	u_0(x)&=1;
\end{split}\right.\end{equation}
where the Laplace operator $\Delta$ acts on the variable $x\in\R^d$,
and $\eta:=\{\eta_t(x)\}_{t\ge0,x\in\R^d}$ is a centered generalized
Gaussian random field with covariance measure
\begin{equation}
	\Cov[\eta_t(x)\,,\eta_s(y)] = \delta_0(t-s)f (x-y)\qquad[s,t\ge0,\, x,y\in\R^d],
\end{equation}                          
for a positive-definite bounded and continuous function $f:\R^d\to\R_+$. 

By comparison to the classical situation of Brownian heat baths,
we can think of the solution to \eqref{E:SHE:color} as describing
Brownian motion in a random environment.

When $d=3$, a variation
of equations of the form \eqref{E:SHE:color} was introduced in 
cosmology in order to describe the large-scale structure of the universe.
Here are some more details:
It is believed that,\footnote{See Albeverio et al \cite{AMS}
and its detailed references to the physics literature, in particular, to
the pioneering work of Zel'dovich and his collaborators.}
after we make a standard change of
variables to remove mathematically-uninteresting physical constants,
the velocity field $\vec{v}$ of galaxy masses
approximately solves the $3$-dimensional stochastic PDE,\footnote{Here,
$\nabla{\vec{v}}$ denotes the matrix of all first derivatives of the coordinate
functions of $\vec{v}$.}
\begin{equation}
	\frac{\partial}{\partial t} \vec{v}_t(x) = \tfrac12(\Delta \vec{v}_t)(x)
	- \left(\vec{v}_t(x)\right)^{\sf T}\nabla\vec{v}_t(x)+ \nabla\Phi_t(x),
\end{equation}
for $(t\,,x)\in(0\,,\infty)\times\R^3$, subject to the following:
\begin{equation}
	\text{\rm curl}\left( \vec{v}_t(x)\right) =0,\quad
	\vec{v}_0(x) = -\nabla\psi(x);
\end{equation}
for all $t>0$ and $x\in\R^3$. The external field $\Phi$ is a scalar field 
and believed to be random, and the initial field $\psi$ may or may not be random.
Now we apply a formal Hopf--Cole transform and posit that
\begin{equation}
	\vec{v}_t(x) = -\nabla\log \phi_t(x),
\end{equation}
for a scalar field $\phi$. It is then easy to see that, if $\psi$ and $\Phi$
were smooth, then $\phi$ would solve
\begin{equation}\left[\begin{split}
	\dot{\phi}_t(x)&=\tfrac12(\Delta \phi_t)(x)+ \phi_t(x)\Phi_t(x) 
		\quad \text{for $t>0, x\in\R^3$};\\
	\phi_0(x)&=\e^{\psi(x)}.
\end{split}\right.\end{equation} 
As far as we know, 
there is no general agreement on what the external field $\Phi$ should be,
though the simplest form of the ``big-bang theory'' might suggest that $\psi=\delta_0$, after
a suitable relabeling of $\R^3$.
Our stochastic PDE \eqref{E:SHE:color} is this  equation
in the particular case that $\Phi\equiv\eta$ and $\psi\equiv0$.

\subsection{The  main result}
It is a classical fact that our correlation function $f$ is uniformly
continuous and maximized at the origin; i.e.,
\begin{equation}\label{ff}
	f(z) \le f(0)\qquad\text{for all $z\in\R^d$}.
\end{equation}
Therefore, in order to avoid trivialities we will always assume that 
\begin{equation}\label{f(0)>0}
	f(0)>0.
\end{equation}
Indeed, if $f(0)$ were zero, then $f\equiv 0$ and hence $\eta\equiv 0$.
In that case, the solution to \eqref{E:SHE:color} is $u_t(x)\equiv1$,
trivially.

Let $\hat{f}$ denote the distributional Fourier transform of $f$. 
Because $f$ is assumed to be positive definite, $\hat{f}$ is a 
positive distribution. That is, $\hat{f}$ is a tempered Borel measure on $\R^d$
thanks to the Riesz representation theorem. Furthermore, the Parseval identity
shows that
\begin{equation}
	\int_{\R^d}|\hat\varphi(x)|^2\,\hat{f}(\d x)
	= (2\pi)^d\int_{\R^d} (\varphi*\tilde{\varphi})(z)f(z)\,\d z
	\le  (2\pi)^d f(0)\|\varphi\|_{L^1(\R^d)}^2,
\end{equation}
for all rapidly-decreasing test functions $\varphi:\R^d\to\R$,
where $\tilde\varphi$ denotes the \emph{reflection}
of $\varphi$; that is,
\begin{equation}\label{reflection}
	\tilde{\varphi}(x):=\varphi(-x)\qquad\text{for all $x\in\R^d$.}
\end{equation}
We replace $\varphi$ by $\varphi_\epsilon$---where 
$\{\varphi_\epsilon\}_{\epsilon>0}$ is
an approximate identity built from functions in $\cS(\R^d)$---%
and let $\epsilon\downarrow0$ in order to conclude that 
$\hat{f}$ is in fact a finite Borel measure on $\R^d$. In particular,
this shows that
\begin{equation}\label{E:Dalang}
	\int_{\R^d} \frac{\hat{f} (\d z)}{1+\|z\|^2} < \infty.
\end{equation}
Thanks to the theory of Dalang \cite{Dalang}, 
condition \eqref{E:Dalang} implies the existence of
a predictable mild solution to \eqref{E:SHE:color}. Moreover, that
solution is unique among all predictable solutions that satisfy
\begin{equation}
	\sup_{x\in\R^d}\sup_{t\in[0,T]}\E\left(|u_t(x)|^k\right)<\infty
	\quad\text{for all $T>0$ and $k\ge 2$}.
\end{equation}

This solution to \eqref{E:SHE:color} 
can also be written in mild form as the a.s.-unique 
solution to the random integral equation, 
\begin{equation}\label{mild:color}
	u_t(x)=(p_t\ast u_0)(x)+\int_{(0,t)\times\R} p_{t-s}(x-y)u_s(y) \,\eta(\d s \, \d y),
\end{equation} 
where the stochastic integral is understood in the sense of Walsh \cite{Walsh},
and now $p$ denotes the natural $d$-dimensional generalization to \eqref{p}:
\begin{equation}
	p_t(x):=\frac{\e^{-\|x\|^2/2t}}{(2\pi t)^{d/2}}\qquad
	[t>0,x\in\R^d].
\end{equation}

Formally speaking,  we can let $d:=1$, $\sigma(z):=z$, and $f(z):=\delta_0(z)$
to see that \eqref{E:SHE:color} is [in this case] one possible extension of 
\eqref{eq:pam} to higher dimensions. Although $\delta_0$ is not a continuous
and bounded function, it is an appropriate limit of such functions. As such,
one can derive \eqref{eq:pam}---in this case---using a limiting procedure
from the solution to \eqref{E:SHE:color} with a suitable
approximate identity $\{f_\epsilon\}_{\epsilon>0}$ in place of $f$.
See Bertini and Cancrini \cite{BertiniCancrini} for the details.

There is a good way to construct examples
of positive-definite continuous and bounded functions $f:\R^d\to\R_+$ as follows:
\begin{equation}\label{f=h*h}
	f := h*\tilde{h},
\end{equation}
where $h\in L^2(\R^d)$ is a nonnegative fixed function, 
and $\tilde{h}$ denotes the reflection of $h$
[see \eqref{reflection}]. Thanks to \eqref{ff}, we can see that
\begin{equation}
	f(0) = \sup_{z\in\R^d} f(z) = \|h\|_{L^2(\R^d)}^2.
\end{equation}
We will restrict attention to such correlation functions
$f$ only. In fact, we concentrate on a slightly-smaller class of correlation
functions still. Namely, we will consider only correlation functions $f$ that satisfy \eqref{f=h*h}
for a nonnegative $h\in L^2(\R^d)$ that satisfies the following:
\begin{equation}\label{cond:h}
	\limsup_{n\to\infty}\frac{1}{\log n}\log\left(
	\int_{\|z\|>n} [h(z)]^2 \d z\right) <0.
\end{equation}

It is possible to write a Feynman--Kac type representation of
the solution to \eqref{E:SHE:color}. That representation implies
readily that $u_t(x)>0$ a.s.\ for all $t>0$ and $x\in\R^d$. 
With this remark in place, we have the following, which is
the main result of this section.

\begin{theorem}\label{th:pam:L2color}
	Consider the SPDE \eqref{E:SHE:color} where the spatial correlation
	function $f$ of the noise satisfies \eqref{f(0)>0}
	and \eqref{f=h*h} for 
	a nonnegative function $h\in L^2(\R^d)$ that satisfies \eqref{cond:h}.
	Then, with probability one:
	\begin{equation}\label{LIL:color}
		\limsup_{\|x\|\to\infty} \frac{\log u_t(x)}{(\log\|x\|)^{1/2}}
		=\sqrt{2tf(0)d},
	\end{equation}
	and
	\begin{equation}\label{Dim:color}
		\Dimh
		\left\{x\in \R^d: \|x\|>\e^\e, \,
		u_t(x) \ge \e^{\gamma\sqrt{t\log \|x\|}}\right\}
		=d-\frac{\gamma^2}{2f(0)}.
	\end{equation}
\end{theorem}

Recall that for all sets $E$, $\Dimh E<0$
means that $E$ is bounded. It follows easily from this convention that 
the lim sup law \eqref{LIL:color} is a consequence of \eqref{Dim:color}.
Conus et al \cite{CJK} proved that the lim sup in \eqref{LIL:color} is strictly
positive and finite a.s., and the evaluation of the lim sup is contained in the recent
work of Chen \cite{Chen}. We have included \eqref{LIL:color} merely to highlight
the fact that the tall peaks of $u_t(x)$ are of order $\exp\{\gamma\sqrt{t\log\|x\|}\}$
for a constant $\gamma>0$, and hence that \eqref{Dim:color}
is indeed a multifractal description of the tall peaks of $u_t(x)$. 

The remainder of this section contains the proof of Theorem \ref{th:pam:L2color}.

\subsection{Proof of Theorem \ref{th:pam:L2color}: Upper Bound}

The proof of Theorem \ref{th:pam:L2color} proceeds by verifying the
conditions of the general theory of \S\ref{sec:Gen}. Thus, the proof
is divided naturally into two parts: Proof of the dimension upper bound;
and a separate derivation of the dimension lower bound.

The dimension upper bound will be obtained by verifying Condition 
\eqref{cond:UB}.  Our first lemma is essentially a specialization of the proof
of Proposition 4.4 of Conus et al \cite{CJKS}. The only new observation is
that since here $f$ is continuous, the constants in Proposition 4.4
can be computed explicitly. 

\begin{lemma}\label{L:moments:L2}
	For all $t>0$,
	\begin{equation}\label{E:moment:asymp}
      		\lim_{k\rightarrow \infty} \frac{1}{k^2}
      		\log \E\left( |u_t(x)|^k\right) = \frac{f(0)t}{2},
	\end{equation}
	uniformly for all $x\in\R^d$.
\end{lemma}

\begin{proof}
	If $m\ge1$ is an integer and $k\in[m\,,m+1)$, then Jensen's inequality 
	assures us that
	\begin{equation}
		\| u_t(x)\|_{L^m(\Omega)}\le
		\| u_t(x)\|_{L^k(\Omega)} \le \| u_t(x)\|_{L^{m+1}(\Omega)}.
	\end{equation}
	Therefore, it suffices to prove that the lemma holds  where the limit
	is taken over all integers $k\to\infty$.
	
	Proposition 4.4 of Conus et al \cite{CJKS} includes the statement that
	\begin{equation}\label{mom:UB}
		\E\left( |u_t(x)|^k\right) \le \e^{k^2 t f(0)/2},
	\end{equation}
	for all real $t\ge0$, $x\in\R^d$, and integers $k\ge 2$.
	
	We develop a corresponding lower bound by following the proof of
	Proposition 4.4 of \cite{CJKS}, but use the additional hypothesis that
	$f$ is continuous.
	
	According to the Feynmen--Kac formula for the moments of the solution
	to \eqref{E:SHE:color}---see Conus \cite{Conus} and Hu and Nualart 
	\cite{HuNualart}---the
	$k$th moment of the solution to \eqref{E:SHE:color} has the following representation:
	\begin{equation}\label{E:moment:formula}
		\E\left( |u_t(x)|^k\right) =
		\E\left[\exp\left( \tfrac12 \mathop{\sum\sum}\limits_{1\leq i\neq j \leq k} 
		\int_0^t f\left(X_i(s)-X_j(s) \right) \d s\right)\right],
	\end{equation}
	where $\{X_i\}_{i=1}^k$ are independent Brownian motions on $\R^d$. 
	
	Choose and fix some $\varepsilon>0$. There exists $\delta(\varepsilon)>0$ 
	such that $f(x)\geq f(0)-\varepsilon$ whenever $\|x\|\leq \delta(\varepsilon)/2$. 
	Now let us consider the event 
	\begin{equation}
		\Omega_\varepsilon:=\left\{\omega\in\Omega:\
		\max_{1\leq j \leq k} \sup_{s\in[0,t]} \|X_j(s) \| (\omega) \leq 
		\delta(\varepsilon) \right\}.
	\end{equation}
	The moment formula \eqref{E:moment:formula} and the continuity of $f$
	together imply that
	\begin{align}
		\E\left(|u_t(x)|^k\right) 
			&\ge\exp\left\{ \binom{k}{2} (f(0)-\varepsilon)t\right\} 
			\P(\Omega_\varepsilon)\\\notag
		&= \exp\left\{ \binom{k}{2} (f(0)-\varepsilon)t \right\}
			\left[\P\left\{\sup_{s\in [0,t/|\delta(\varepsilon)|^2]} 
			\|X_1(s)\|\leq 1\right\}\right]^k,
	\end{align}
	uniformly for all $t\ge0$ and $x\in\R^d$, and all integers $k\ge 2$.
	This implies that
	\begin{equation}
		\liminf_{\substack{k\rightarrow \infty:\,
		k\in\Z}} k^{-2} \log\E \left(|u_t(x)|^k\right) \ge {(f(0)-\varepsilon)t}/2,
	\end{equation}
	and readily yields the desired lower bound since $\varepsilon>0$
	were arbitrary.
\end{proof}

We can now invert moments, exactly as was
done in the proof of  Theorem 5.2 of Chen \cite{Chen},
in order to deduce the following. 
We skip the proof since it really follows the proof of \cite[Theorem 5.2]{Chen}
almost exactly.

\begin{lemma}\label{L:tailprob:L2}
	Fix $t>0$ and a bounded domain $D\subset\R^d$. 
	Then uniformly for all $x\in\R^d$,
	\begin{align}\notag
		\lim_{z\rightarrow \infty}  z^{-2}  
			\log \P\left\{ \log u_t(x) \ge z\right\}&=
			\lim_{z\rightarrow \infty} z^{-2} 
			\log \P \left\{\log\left[\sup_{y\in Q(x,1)} u_t(y) \right] \ge z \right\}\\
		&=-(2f(0)t)^{-1}.
		\label{E:condUB:L2}
	\end{align}
\end{lemma}

Lemma \ref{L:tailprob:L2} verifies the conditions of Theorem \ref{th:Gen:LIL:UB},
from which we can deduce half of Theorem \ref{th:pam:L2color} readily.

\begin{proof}[Proof of Theorem \ref{th:pam:L2color}: Dimension upper bound]
	
	Lemma \ref{L:tailprob:L2} verifies
	Condition \eqref{cond:UB}  of Theorem \ref{th:Gen:LIL:UB}
	with $X(x):=\log u_t(x)$, $z:=(2\gamma f(0) t \log s)^{1/2}$ and $b=2$.
	It follows from Theorem \ref{th:Gen:LIL:UB}
	that 
	\begin{equation}\label{Dim:L2color:UB}
		\Dimh \left\{ x\in\R^d:\ \|x\|>\e^\e,\
		u_t(x)\ge\e^{\gamma\sqrt{t\log \|x\|}}
		\right\}\le d- \frac{\gamma^2}{2f(0)},
	\end{equation}
	almost surely for every $t,\gamma>0$. This completes the proof
	of the dimension upper bound in Theorem \ref{th:pam:L2color}.
\end{proof}

\begin{remark}
	Let us record also the fact that the preceding proof required only that
	$h\ge 0$ is in $L^2(\R^d)$; the extra regularity condition
	\eqref{cond:h} was not needed, and \eqref{Dim:L2color:UB} holds
	without that extra condition.
\end{remark}

\subsection{A coupling of the noise}
We will have need for a particular construction of $\eta$ that can be found
essentially in the paper by Conus et al
\cite{CJKS}. Let $\xi$ denote a space-time white noise on $\R_+\times\R^d$; that is, $\xi$ is
a centered generalized Gaussian random field with covariance measure
\begin{equation}
	\Cov[\xi_t(x)\,,\xi_s(y)] = \delta_0(t-s)\cdot\delta_0(x-y)
	\qquad\text{[$s,t>0,\, x,y\in\R^d$]}.
\end{equation}
We can construct a cylindrical Brownian motion $B$, using $\xi$, as follows:
For all $\varphi\in L^2(\R^d)$ and $t>0$, define
\begin{equation}
	B_t(\varphi) := \int_{(0,t)\times\R^d} \varphi(y)\,\xi(\d s\,\d y)
	\quad\text{and}\quad
	B_0(\varphi):=0.
\end{equation}
Then $\{B_t(\varphi)\}_{t\ge0,\varphi\in L^2(\R^d)}$ is a centered Gaussian random field
with 
\begin{equation}
	\Cov\left[ B_t(\varphi_1)\,, B_t(\varphi_2)\right] = 
	\min(s\,,t)\cdot (\varphi_1\,,\varphi_2)_{L^2(\R^d)},
\end{equation}
for all $s,t\ge0$ and $\varphi_1,\varphi_2\in L^2(\R^d)$. 
Thus, we see that $\{B_t\}_{t\ge0}$
is a cylindrical Brownian motion on $L^2(\R^d)$ [see Da Prato and Zabczyk
\cite[\S4.3.1]{dPZ}. In particular, we can recover
the space-time white noise $\xi$ from $B$ by noticing that
\begin{equation}
	\xi_t(x) = \frac{\partial^{d+1}}{\partial t\,\partial x_1\,\cdots
	\partial x_d} B_t(x),
\end{equation}
where the derivative is understood in the sense of random linear functionals;
see Chapter 2 of Gel'fand and Vilenkin \cite{GV} for this topic.

We can \emph{subordinate} a large family of Gaussian 
random fields to the cylindrical Brownian motion $\{B_t\}_{t\ge0}$
as follows: For all $\varphi,\psi\in L^2(\R^d)$ we define a new centered 
Gaussian random field $\{B_t^{(\psi)}(\varphi)\}_{t\ge0,\psi,\varphi\in L^2(\R^d)}$
by setting
\begin{equation}
	B^{(\psi)}_t(\varphi) := B_t(\varphi*\tilde{\psi}),
\end{equation}
where we recall $\tilde{\psi}$ denotes the reflection of $\psi$ [see \eqref{reflection}].

The random mapping $(\varphi\,,\psi)\mapsto B^{(\psi)}(\varphi)$
is and linear, and the covariance structure of each Gaussian process $B^{(\psi)}$ is dictated by
\begin{equation}
	\Cov\left[ B^{(\psi)}_t(\varphi_1)\,, B^{(\psi)}_s(\varphi_2)\right]
	= \min(s\,,t)\cdot \left( \varphi_1\,, \varphi_2 * \psi*\tilde{\psi}\right)_{L^2(\R^d)},
\end{equation}
for every $s,t\ge0$ and $\varphi_1,\varphi_2,\psi\in L^2(\R^d)$.
In particular, \eqref{f=h*h} yields
\begin{equation}
	\Cov\left[ B^{(h)}_t(\varphi_1)\,, B^{(h)}_s(\varphi_2)\right]
	= \min(s\,,t)\cdot \left( \varphi_1\,, \varphi_2 * f\right)_{L^2(\R^d)},
\end{equation}
for every $s,t\ge0$ and $\varphi_1,\varphi_2\in L^2(\R^d)$. This is another way
to say that the weak derivative $\partial_t B^{(h)}_t(x)$
is a particular construction of the noise $\eta_t(x)\, \d t$. Since we are interested only in
the law of the solution $u$ to \eqref{E:SHE:color}, and that law is by construction
a function of the law of $\eta$, we may---and will---change probability space if we
have to in order to construct the noise $\eta$ on the new probability space as follows:
\begin{equation}
	\eta_t(x) := \frac{\partial}{\partial t} B_t^{(h)}(x)\qquad(t\ge0,x\in\R^d).
\end{equation}
Thus, we are justified in using the following notation to denote
the Wiener integral $\int_{\R_+\times\R^d}\Phi\,\d\eta$:
\begin{equation}
	\int_{\R_+\times\R^d} \Phi_s(y)\,\eta(\d s\,\d y)
	:= \int_{\R_+\times\R^d}\Phi_s(y)\, \partial_s B^{(h)}_s(y)\,\d y,
\end{equation}
for all nonrandom functions $\Phi\in L^2(\R_+\times\R^d)$.
Thus, in particular, we may---and will---think of the solution to 
the stochastic PDE \eqref{E:SHE:color} as the unique solution to 
the following stochastic integral equation
\begin{equation}\label{mild:color}
	u_t(x) = 1 + \int_{(0,t)\times\R^d} p_{t-s}(y-x)
	u_s(y)\,\partial_s B^{(h)}_s(y)\,\d y,
\end{equation}
for all $t>0$ and $x\in\R^d$, where the stochastic integral is understood
in the sense of Walsh \cite{Walsh} and written using the notation introduced 
earlier.

For us, an advantage of this construction is that, 
in this way,  not only have a construction of $\eta_t(x)$ for our fixed
function $h$, but we have in fact produced a coupling of
$\psi\mapsto \partial_t B^{(\psi)}_t(x)$ which is a linear map and agrees with
$\eta_t(x)$ when $\psi=h$. 

For purposes of comparison, let us mention that
the stochastic differential
$F^{(h)}(\d s\,\d y)$ of Conus et al \cite{CJKS} is the
same thing as our 
mixed random differential $\partial_s B^{(h)}_s(y)\,\d y$.

\subsection{Proof of Theorem \ref{th:pam:L2color}: Lower Bound}

In order to prove the lower bound on the dimension in Theorem \ref{th:pam:L2color}
we plan to verify Condition \eqref{cond:LB} using an approach that has
the same flavor as the proof of the already-developed
lower bound of Theorem \ref{th:pam:LIL}. There are some nontrivial differences in
the proofs, however. Most notably,
since the spatial correlation of the noise of the present section is
not in general  0 even when $x$ and $y$ are very far apart, we need
to do something more. To combat this issue, we first approximate 
the function $h$ by a compactly-supported function $h_\beta$,
and then follow the approach used in  the proof of the lower bound of 
Theorem \ref{th:pam:LIL}. Our proof will follow the ideas of Conus et al
\cite{CJKS} loosely.

Define, for all $x\in\R^d$ and $\beta\geq 1$,
\begin{equation}\label{h:beta}
	h_\beta(x) := h(x)\hat\varrho_\beta(x)
	\quad\text{where}\quad
	\hat\varrho_\beta(x):=\prod_{j=1}^d \left(1-\frac{|x_j|}{\beta}\right)^+.
\end{equation}

Every function $h_\beta$ is in $L^2(\R^d)$ 
and has compact support. In addition,
$h_\beta$ converges to $h$ pointwise as $\beta\to\infty$. 
Therefore, we ought to be able to construct approximations of $u_t(x)$ 
by first approximating the noise $\partial_t B_t^{(h)}(x)\,\d x$ by the
noise $\partial_t B_t^{(h_\beta)}(x)\,\d x$  of the previous subsection.

In order to simplify notation, let us write
for every $x:=(x_1,\dots, x_d)\in \R^d$, $t>0$, and $\beta\ge 0$, 
\begin{equation} 
	\mathcal{I}_t(x;\beta):=
	(0\,,t)\times \left[x_1-\beta\sqrt{t}\,, x_1+\beta\sqrt t\right]\times \cdots \times 
	\left[x_d-\beta\sqrt t\,, x_d+\beta\sqrt t\right].
\end{equation}
 
We can follow the lead of Conus et al \cite{CJKS},
and choose and fix $\beta>0$, and consider the solution $u^{(\beta)}$
to the following stochastic integral equation:
\begin{equation}\label{u:beta:color}
	u_t^{(\beta)}(x):=1+\int_{\mathcal{I}_t(x;\beta)} 
	p_{t-s}(y-x)\, u_s^{(\beta)}(y)\, \partial_s B^{(h_\beta)}_s(y)\,\d y.
\end{equation}

The following was  pointed out in \cite[\S5]{CJKS} without proof.
\begin{proposition}\label{pr:exist:color}
	The stochastic integral equation \eqref{u:beta:color}
	has a predictable solution that is unique among all
	such solutions that satisfy the following for all real numbers
	$k\ge 2$:
	\begin{equation}
		\sup_{x\in\R^d} \E\left( \left| u^{(\beta)}_t(x) \right|^k\right)
		\le 2\e^{8k^2f(0)t}\qquad\text{for all $\beta>0$ and $t\ge0$}.
	\end{equation}
\end{proposition}

We will need some of the ingredients of that proof, and the details
are not included in \cite{CJKS}; therefore, let us hash out
a few of the standard details. 

Keep $\beta>0$ fixed, and define for all $t\ge 0$ and $x\in\R^d$,
\begin{equation}
	u_t^{(\beta,0)}(x) :=1.
\end{equation}
Then we define iteratively the random field $u_t^{(\beta,m)}(x)$
as follows: For all integers $m\ge 0$, reals $t>0$, and $x\in\R^d$,
\begin{equation}\label{eq:coupling:u}
	u_t^{(\beta,m+1)}(x):=1+\int_{\mathcal{I}_t(x;\beta)} 
	p_{t-s}(y-x)\, u_s^{(\beta,m)}(y) \,\partial_s B^{(h_\beta)}_s(y)\,\d y.
\end{equation}

The proof of Proposition \ref{pr:exist:color} requires two \emph{a priori}
bounds. The first controls the moments of the Picard iterates.

\begin{lemma}\label{mom:u:beta,m}
	Uniformly for all $t\ge0$, $x\in\R^d$, $k\ge 2$, and
	integers $m\ge 0$,
	\begin{equation}
		\E\left( \left| u^{(\beta,m)}_t(x)\right|^k\right) \le 2
		\e^{8k^2 f(0) t}.
	\end{equation}
\end{lemma}

\begin{proof}
	Throughout, we hold $\beta>0$ fixed and
	define for all $t>0$, $m\ge 0$, and $x\in\R^d$,
	\begin{equation}\label{Xm}
		X^{(m)}_t(x):=
		\int_{\mathcal{I}_t(x;\beta)} 
		p_{t-s}(y-x)\, u_s^{(\beta,m)}(y) \,\partial_s B^{(h_\beta)}_s(y)\,\d y,
	\end{equation}
	if and when the stochastic integral is defined in the sense of Walsh
	\cite{Walsh}. As an \emph{a priori} part of this proof we will derive
	moment bounds for $X^{(m)}_t(x)$. 
	
	Let 
	\begin{equation}
		f_\beta := h_\beta * \tilde{h}_\beta,
	\end{equation}
	where $h_\beta$ was defined in \eqref{h:beta}.
	The starting point is that a suitable
	form of the Burkholder--Davis--Gundy inequality 
	\cite[Theorem B.1]{cbms} implies
	that for all real numbers $k\ge 2$,
	\begin{equation}\label{T1}\begin{split}
		\left\| X^{(m)}_t(x)\right\|_k^2 &\le  4k \left\| \int_{\mathcal{I}_t(x;\beta)}
			\mathcal{T}(s\,,y\,,z)\, \d s\d z\d y\right\|_{k/2}\\
		&\le 4k \int_0^t\d s\int_{\R^d}\d y\int_{\R^d}\d z\
			\| \mathcal{T}(s\,,y\,,z)\|_{k/2},
	\end{split}\end{equation}
	where we have used Minkowski's inequality in the last line, and
	\begin{equation}\begin{split}
		\mathcal{T}(s\,,y\,,z) :=&\ p_{t-s}(y-z) p_{t-s}(z-x)\left| 
			u^{(\beta,m)}_s(y) u^{(\beta,m)}_s(z) \right|f_\beta(y-z) \\
		\le&\ p_{t-s}(y-z) p_{t-s}(z-x)\left| 
			u^{(\beta,m)}_s(y) u^{(\beta,m)}_s(z) \right|f(0),
	\end{split}\end{equation}
	thanks to \eqref{ff} and the elementary fact that
	$h_\beta\le h$, whence $f_\beta\le f$.
	In particular, by the Cauchy--Schwarz inequality,
	\begin{equation}
		\|\mathcal{T}(s\,,y\,,z)\|_{k/2}
		\le \e^{2\alpha s} \left[\cN_{\alpha,k}\left( u^{(\beta,m)} \right) \right]^2
		p_{t-s}(y-z) p_{t-s}(z-x)f(0),
	\end{equation}
	for every $\alpha>0$, where for all space-time random
	fields $\Psi:=\{\Psi_s(y)\}_{s\ge 0,y\in\R^d}$,
	\begin{equation}\label{M:alpha}
		\cN_{\alpha,k}(\Psi) := \sup_{y\in\R^d}
		\sup_{s\ge 0}\left[ \e^{-\alpha s} \left\| \Psi_s(y)\right\|_k
		\right].
	\end{equation}
	We plug the preceding into \eqref{T1} to see that for all $\alpha>0$,
	\begin{equation}\begin{split}
		\left\| X^{(m)}_t(x)\right\|_k^2
			&\le 4k \left[\cN_{\alpha,k}\left( u^{(\beta,m)} \right) \right]^2f(0)\cdot
			\int_0^t \e^{2\alpha s}\,\d s\\
		&\le \frac{2kf(0)\e^{2\alpha t}}{\alpha}
			\left[\cN_{\alpha,k}\left( u^{(\beta,m)} \right) \right]^2f(0).
	\end{split}\end{equation}
	Multiply both sides by $\exp(-2\alpha t)$ and maximize over $x\in\R^d$
	to see that
	\begin{equation}
		\cN_{\alpha,k}\left( X^{(m)}\right) \le \cN_{\alpha,k}\left( u^{(\beta,m)} \right) 
		\sqrt{\frac{2kf(0)}{\alpha}}.
	\end{equation}
	Since $\Psi\mapsto \cN_{\alpha,k}(\Psi)$ is a bona fide norm on random fields,
	it follows from \eqref{eq:coupling:u} and the triangle inequality that
	\begin{equation}
		\cN_{\alpha,k}\left( u^{(\beta,m+1)}\right) \le 
		1 +  \cN_{\alpha,k}\left( u^{(\beta,m)} \right) 
		\sqrt{\frac{2 k f(0)}{\alpha}}.
	\end{equation}
	The preceding is valid for all $\alpha>0$. We now select $\alpha:=8kf(0)$
	in order to see that, for this special choice,
	\begin{equation}
		\cN_{8kf(0),k}\left( u^{(\beta,m+1)}\right) \le 
		1 +  \tfrac12 \cN_{8kf(0),k}\left( u^{(\beta,m)} \right).
	\end{equation}
	Since $\cN_{\alpha,k}(u^{(\beta,0)})=1$ for all $\alpha>0$, induction implies
	that
	\begin{equation}
		\cN_{8kf(0),k}\left( u^{(\beta,m+1)}\right) \le 2.
	\end{equation}
	This is another way to write the lemma.
\end{proof}

Next we state and prove the second \emph{a priori} bound that is
required for the proof of Proposition \ref{pr:exist:color}.

\begin{lemma}\label{mom:u:beta,m+1-m}
	Uniformly for all $t\ge0$, $x\in\R^d$, $k\ge 2$, and
	integers $m\ge 0$,
	\begin{equation}
		\E\left( \left| u^{(\beta,m+1)}_t(x) - u^{(\beta,m)}_t(x)\right|^k\right)
		\le \left( \frac{3}{2^m}\right)^k \e^{8k^2f(0)t}.
	\end{equation}
\end{lemma}

\begin{proof}
	The proof is similar to that of Lemma \ref{mom:u:beta,m}.
	Recall $X^{(m)}$'s from \eqref{Xm}. By the Burkholder--Davis--Gundy
	inequality \cite[Theorem B.1]{cbms},
	\begin{equation}\begin{split}
		\left\| X^{(m+1)}_t(x) -X^{(m)}_t(x) \right\|_k^2
			&\le 4k\left\| \int_{\mathcal{I}_t(x;\beta)}
			\mathcal{Y}(s\,,y\,,z)\, \d s\d z\d y\right\|_{k/2}\\
		&\le 4k\int_0^t\d s\int_{\R^d}\d y\int_{\R^d}\d z\ 
			\|\mathcal{Y}(s\,,y\,,z)\|_{k/2},
	\end{split}\end{equation}
	where
	\begin{equation}\begin{split}
		\mathcal{Y}(s\,,y\,,z) &:= p_{t-s}(y-z) p_{t-s}(z-x)f(0) \\
		&\quad \times
			\left| u^{(\beta,m+1)}_s(y)-u^{(\beta,m)}_s(y) \right|
			\cdot \left|u^{(\beta,m+1)}_s(z)-u^{(\beta,m)}_s(z) \right|.
	\end{split}\end{equation}
	Recall $\cN_{\alpha,k}:=\cN_{\alpha,k}$ from \eqref{M:alpha}. Since
	\begin{equation}
		\cN_{8kf(0),k}\left( u^{(\beta,m+1)} - u^{(\beta,m)}\right)
		= \cN_{8kf(0),k}\left( X^{(m+1)} - X^{(m)}\right),
	\end{equation}
	we can easily adapt the proof of Lemma \ref{mom:u:beta,m} to see that
	\begin{equation}\label{recur:color}
		\cN_{8kf(0),k}\left( u^{(\beta,m+1)} - u^{(\beta,m)}\right)
		\le \tfrac12 
		\cN_{8kf(0),k}\left( u^{(\beta,m)} - u^{(\beta,m-1)}\right).
	\end{equation}
	Because $u^{(\beta,0)}_t(x)=1$, Lemma \ref{mom:u:beta,m} implies that
	\begin{equation}
		\cN_{8kf(0),k}\left( u^{(\beta,1)}-u^{(\beta,0)}\right) \le
		\cN_{8kf(0),k}\left( u^{(\beta,1)}\right) + 
		\cN_{8kf(0),k}\left( u^{(\beta,0)}\right)\le 3.
	\end{equation}
	Therefore,
	$\cN_{8kf(0),k} ( u^{(\beta,m+1)} - u^{(\beta,m)} )
	\le 3\cdot 2^{-m}$, which is another way to state the lemma.
\end{proof}

We are now ready to complete the proof of Proposition \ref{pr:exist:color}.

\begin{proof}[Proof of Proposition \ref{pr:exist:color}]
	Recall $\cN_{\alpha,k}$ from \eqref{M:alpha}.
	Lemmas \ref{mom:u:beta,m} and \ref{mom:u:beta,m+1-m}
	together guarantee the existence of a predictable random field
	$u^{(\beta)}$ such that 
	\begin{equation}
		\cN_{8kf(0),k}\left( u^{(\beta)}\right)\le 2
		\quad\text{and}\quad
		\lim_{m\to\infty}
		\cN_{8kf(0),k}\left( u^{(\beta,m)} - u^{(\beta)}\right)=0.
	\end{equation}
	The proof of Lemma  \ref{mom:u:beta,m+1-m} also implies that
	\begin{equation}
		\lim_{m\to\infty} \cN_{8kf(0),k}\left( X^{(m)} - X\right)=0,
	\end{equation}
	where $X^{(m)}$ was defined in \eqref{Xm}, and
	\begin{equation}
		X_t(x):=
		\int_{\mathcal{I}_t(x;\beta)} 
		p_{t-s}(y-x)\, u_s^{(\beta)}(y) \,\partial_s B^{(h_\beta)}_s(y)\,\d y.
	\end{equation}
	These remarks together show that $u^{(\beta)}$ solves \eqref{u:beta:color}.
	Uniqueness is similar; in fact, the argument that
	led to \eqref{recur:color} can be re-iterated in order to imply that
	if $v$ were any other predictable solution to \eqref{u:beta:color}
	that satisfies $\cN_{8kf(0),k}(v)<\infty$, then
	\begin{equation}
		\cN_{8kf(0),k}\left(v-u^{(\beta)}\right) \le \tfrac12
		\cN_{8kf(0),k}\left(v-u^{(\beta)}\right),
	\end{equation}
	and hence $\cN_{8kf(0),k}(v-u^{(\beta)})=0$.
\end{proof}

Now that we have justified the existence of a good solution
to \eqref{u:beta:color} we can establish that $u\approx u^{(\beta)}$
when $\beta$ is large.

\begin{lemma}[A coupling lemma]\label{lem:u-ubeta:L2color}
	There exists a finite constant $K>1$  such that  
	for all real numbers $t>0$ and $\lambda>1$, and 
	all integers $\beta\ge 1$, 
	\begin{equation}\label{eq:L2:coupling}
		\P\left\{ \left| u_t(x) - u^{(\beta,\lfloor\log\beta\rfloor+1)}_t(x)\right|>\lambda\right\}
		\le K\exp\left( -\frac{(\log\beta+\log \lambda)^{2}}{K t}\right),
	\end{equation}
	uniformly over all $x\in\R^d$.
\end{lemma}

\begin{proof}
	According to Lemma 5.3 of  Conus et al \cite{CJKS}, 
	there exist  finite constants $c>0$ and $b\in (0\,,4)$ such that
	\begin{equation}
		\sup_{x\in\R}
		\E\left(\left|u_t(x)-u_t^{(\beta,\lfloor\log\beta\rfloor+1)}(x)\right|^k\right) 
		\leq  c \e^{ ck^2 t-b k\log\beta},
	\end{equation}
	valid uniformly for all $x\in\R^d$ and
	all real numbers $k,\beta\ge 2$ and $t>0$.
	This bound and Chebyshev's inequality together yield the following: Uniformly
	for all reals $t>0$ and $k,\lambda,\beta\ge 2$,  and $x\in\R^d$, 
	\begin{equation}
		\P\left\{ \left|u_t(x)-u_t^{(\beta,\lfloor\log\beta\rfloor+1)}(x)\right|>\lambda \right\}
		\le c\e^{ ck^2 t-(\log\lambda +b\log\beta)k}.
	\end{equation}
	The preceding readily implies the lemma, after we optimize over $k\ge 2$.
\end{proof}

For every $x,y\in\R^d$, let us define 
\begin{equation}\label{eq:distance}
	D(x\,,y):=\min_{1\leq l\leq d}|x_l-y_l|,
\end{equation}
where we recall $|\,\cdots|$ denotes the $\ell^\infty$-norm
on $\R^d$. The following lemma is due to Conus, Joseph, Khoshnevisan,
and Shiu.

\begin{lemma}[Conus et al
	\protect{\cite[Lemma 5.4]{CJKS}}]\label{lem:L2color:ind}
	Suppose that  $t>0$ is a real number,
	$\beta\ge 1$ is an integer, and $x^{(1)},\ldots,x^{(m)}$
	are points in $\R^d$ such that
	\begin{equation}
		D\left(x^{(i)}\,,x^{(j)}\right)>2\beta
		\left( \lfloor\log\beta\rfloor+1\right)\left(1+\sqrt{t}\right)
		\quad\text{when }1\le i\neq j\le m.
	\end{equation}
	Then, 
	\begin{equation}
		u^{(\beta,\lfloor\log\beta\rfloor+1)}_t\left(x^{(1)}\right),\ldots,
		u^{(\beta,\lfloor\log\beta\rfloor+1)}_t\left(x^{(m)}\right)
	\end{equation}
	are independent random variables.
\end{lemma}

 We can now verify the dimension lower bound of Theorem \ref{th:pam:L2color}.
 
\begin{proof}[Proof of Theorem \ref{th:pam:L2color}: Dimension lower bound]
	Choose and fix a time variable $t>0$.
	We will appeal to Theorem \ref{th:Gen:LIL:LB},
	specifically to the general
	theory of Section \ref{sec:Gen}, using the identifications $X(x):=\log u_t(x)$,
	$b:=2$, and $S(x):=\exp(x)$.
	
	Let us fix a real number $\delta\in(0\,,1)$ 
	and consider an arbitrary collection $\{x^{(i)}\}_{i=1}^{m^d}$ of points 
	that satisfy the following: 
	\begin{itemize}
		\item[(a)] $x^{(i)}\in \cS_n\subset \R^d$ for all $1\le i\le m^d$; and 
		\item[(b)] $D (x^{(i)}\,,x^{(j)} )\ge\exp(\delta n)$ whenever $1\le i\neq j\le m^d$. 
	\end{itemize}
	From now on, we set
	\begin{equation}
		\beta:=\exp( n^{2/3})
		\quad\text{and}\quad
		Y_j :=  Y_{j,n} :=
		\log\left( u^{(\beta,\lfloor\log\beta\rfloor+1)}_t\left(x^{(j)}\right)\right),
	\end{equation}
	for all $1\le i\le m^d$.
	We might observe that there exists $N:=N_t>0$ such that 
	\begin{equation}
		2\left(\lfloor\log\beta\rfloor+1\right)\beta\left(1+\sqrt{t}\right)<\e^{\delta n}
		\quad\text{for all $n\geq N$.}
	\end{equation}
	Therefore, according to Lemma
	\ref{lem:L2color:ind}, $Y_1,\cdots,Y_{m^d}$ are independent whenever $n\geq N$
	; and Lemma
	\ref{lem:u-ubeta:L2color} ensures that
	\begin{equation}
		\max_{1\le j\le m}
		\P\left\{\left|S\left( X(x^{(j)} \right) - S(Y_j)\right| >1 \right\} \le K\exp\left( -
		\frac{n^{4/3}}{K t}\right).
	\end{equation}
	Since the constant $K$ does not depend on the choice of $x^{(1)},\ldots,x^{(m)}\in\cS_n$,
	we have shown that the coupling 
	Condition \eqref{cond:LB}  holds, with room to spare. Therefore,
	Theorem \ref{th:Gen:LIL:LB} implies that a.s.,
	\begin{equation}
		\Dimh
		\left\{x\in \R^d: \|x\|>\e^\e, \,
		u_t(x) \ge \e^{\gamma\sqrt{t\log \|x\|}}\right\}
		\ge d-\frac{\gamma^2}{2f(0)},
	\end{equation}
	for all $\gamma>0$. In light of the already-proved upper bound this
	completes our proof of Theorem \ref{th:pam:L2color}.
\end{proof}

\noindent\textbf{Acknowledgement.} We thank Professor Gregory Lawler
heartily for his many insightful remarks, questions, and comments 
that ultimately led us to this enjoyable research problem.
\begin{spacing}{1}
\begin{small}
\end{small}\end{spacing}
\vskip.1in

\begin{small}
\noindent\textbf{Davar Khoshnevisan \&\ Kunwoo Kim}\\
	(\texttt{davar@math.utah.edu}
	\&\ \texttt{kkim@math.utah.edu})\\
\noindent Dept.\ Mathematics, University of Utah,
		Salt Lake City, UT 84112-0090\\

\noindent\textbf{Yimin Xiao}\\
(\texttt{xiao@stt.msu.edu})\\
\noindent Dept.\  Statistics \&\ Probability, 
	Michigan State University, East Lansing, MI 48824

\end{small}


\begin{thebibliography}{99}
%
\bibitem{AMS} Albeverio, Sergio, Stanislav Molchanov, and Donatas Surgailis.
	Stratified structure of the universe and Burgers' equation --- a probabilistic approach,
	{\it Probab.\ Th.\ Rel.\ Fields} {\bf 100} (1994) 457--484.
%
\bibitem{AlbinChoi} Albin, J. M. P.  and H. Choi.
	A new proof of an old result by Pickands,
	{\it Electr.\ Comm.\ in Probab.}\ {\bf 15} (2010) 339--345.
%
\bibitem{BarlowTaylor1} Barlow, M. T., and S. J. Taylor.
	Fractional dimension of sets in discrete spaces,
	{\it J. Phys. A}\ (3) {\bf 64} (1989) 2621--2626.
%
\bibitem{BarlowTaylor} Barlow, Martin T., and S. James Taylor.
	Defining fractal subsets of $\mathbb{Z}^d$,
	{\it Proc.\ London Math.\ Soc.}\ (3) {\bf 64} (1992) 125--152.
%
\bibitem{BertiniCancrini}
	Bertini, Lorenzo and Nicoletta Cancrini.
	The stochastic heat equation: Feynman--Kac formula and
	intermittence, {\it J. Statist.\ Physics} {\bf 78}{\it (5/6)}  (1994)
	1377--1402.
%
\bibitem{BorCor} Borodin, Alexei and Corwin, Ivan,
	Moments and Lyapunov exponents for the parabolic Anderson model,
	\emph{Ann.\ Appl.\ Probab.}\ {\bf 24}{\it (33)}  (2014)   1172--1198.
%
\bibitem{CM} Carmona, Ren\'e A. and S. A. Molchanov.
	Parabolic Anderson Problem and Intermittency,
	\emph{Memoires of the Amer.\ Math.\ Soc.} {\bf 108},
	American Mathematical Society, Rhode Island, 1994.
%
\bibitem{Chen}
	Chen, Xia.
	Spatial asymptotics for the parabolic Anderson models with
	generalized time-space Gaussian noise, submitted, 2014.
%
\bibitem{Chen-book} 
	Chen, Xia.
	{\it Random Walk Intersections: Large Deviations and Related Topics}, 
	American Mathematical Society,
	Providence, RI, 2010.
%
\bibitem{CollelaLanford}
	Collela, Phillip, and Oscar E. Lanford.
	{\it Appendix: Sample Field behavior for the free  Markov random field},
	In: ``Constructive Quantum Field Theory'' (G. Velo and A. S. Wightman, ed's)
	Lecture Notes in Physics,
	Vol.\ 25, pp.\ 44--70, 1973.
%
\bibitem{Conus} Conus, Daniel.
	Moments for the parabolic Anderson model: On a result of Hu and Nualart,
	{\it Comm.\ Stoch.\ Analysis} {\bf 7}{\it (1)} (2013)  125--152.
%
\bibitem{CJK}
	Conus, Daniel, Mathew Joseph and Davar Khoshnevisan. 
	On the chaotic character of the stochastic heat equation, before the onset of intermittency,
	{\it Ann.\ Probab.}\ {\bf 41}{\it (3B)}   (2013) 2225--2260.
	
%
\bibitem{CJK-islands}
	Conus, Daniel, Mathew Joseph and Davar Khoshnevisan. 
	Correlation-length bounds, and estimates for intermittent islands in parabolic SPDEs,
	{\it Electr.\ J. Probab.}\ {\bf 17}{\it (102)} (2013) (15 pp).
%
\bibitem{CJKS}
	Conus, Daniel, Mathew Joseph, Davar Khoshnevisan, and Shang-Yuan Shiu.
	On the chaotic character of the stochastic heat equation, II,
	{\it Probab.\ Th.\ Rel.\ Fields} {\bf 156} (2013) 483--533.
%
\bibitem{dPZ} Da Prato, Giuseppe and Jerzy Zabczyk.
	{\it Stochastic Equations in Infinite Dimensions},
	Cambridge University Press, Cambridge, UK, 1992.
%
\bibitem{Dalang} Dalang, Robert C.
	Extending the martingale measure stochastic integral with
	applications to spatially homogeneous s.p.d.e.'s,
	{\it Electron. J. Probab.}\ {\bf 4}{\it (6)} (1999) 29 pp.\
	(electronic). 
%
\bibitem{Minicourse} Dalang, Robert, Davar Khoshnevisan,
	Carl Mueller, David Nualart, and Yimin Xiao.
	\emph{A Minicourse in Stochastic Partial Differential
	Equations} (2006).
	In: Lecture Notes in Mathematics, vol.\  1962
	(D. Khoshnevisan and F. Rassoul--Agha, editors)
	Springer--Verlag, Berlin, 2009.
%
\bibitem{FA:whitenoise} {Foondun, Mohammud}, and
	{Davar Khoshnevisan}. 
	Intermittence and nonlinear stochastic partial differential equations,
	\emph{Electr.\ J. Probab.} {\bf 14}{\it (2)} (2009) 548--568.
%
\bibitem{GV}
	Gel'fand, I. M. and N. Ya. Vilenkin.
	\emph{Generalized Functions}, Vol. 4, Academic Press 
	[Harcourt Brace Jovanovich Publishers], New York,
	1964 [1977], Applications of harmonic analysis, Translated from the Russian
	by Amiel Feinstein. 
%
\bibitem{GT} Gibbon, J. D. and E. S. Titi.
	Cluster formation in complex multi-scale systems,
	{\it Proc.\ R. Soc.\ A} {\bf 461}  (2005) 3089--3097.
%
\bibitem{Hairer} Hairer, Martin.
	Solving the KPZ equation, 
	{\it Ann.\ Math.}\ {\bf 178}{\it (2)} (2013) 559--664.
%
\bibitem{Harper} Harper, Adam J.
Pickand's constant $H_\alpha$ does not equal   $1/\Gamma(1/\alpha)$, for small $\alpha$.  
Available at {\it http://arxiv.org/pdf/1404.5505v1.pdf}
%
\bibitem{HuNualart} Hu,Yaozhong, and David Nualart.
	Stochastic heat equation driven by fractional noise and local time,
	{\it Probab.\ Th.\ Rel.\ Fields} {\bf 143}{\it (1--2)} (2009) 285--328.
%
\bibitem{JKM} Joseph, Mathew, Davar Khoshnevisan, and Carl Mueller.
	Strong invariance and noise comparison principles for some
	parabolic SPDE, submitted, 2014, preprint available at
	\url{http://arxiv.org/abs/1404.6911}.
%
\bibitem{Kardar} Kardar, Mehran.
	Replica Beth ansatz studies of two-dimensional interfaces with quenched
	random impurities,
	{\it Nucl.\ Physics} {\bf B290} [FS20] (1987) 582--602.
%
\bibitem{KPZ} Kardar, Mehran, Parisi, Giorgio, and Zhang,
	Yi-Cheng. Dynamic scaling of growing interfaces,
	{\it Phys.\ Rev.\ Let.}\ {\bf 56} (1986) 889--892.
%
\bibitem{KZ} Kardar, Mehran and Yi-Cheng Zhang,
	Scaling of directed polymers in random media,
	\emph{Phys.\ Rev.\ Lett.}\ {\bf 58}{\it (20)}  (1987) 2087--2090.
%
\bibitem{cbms} Khoshnevisan, Davar.
	{\it Analysis of Stochastic Partial Differential Equations},
	CBMS Regional Conference Series in Mathematics, 119. 
	American Mathematical Society, Providence, RI, 2014.
%
\bibitem{Motoo} Motoo, Minoru.
	Proof of the law of the iteated logarithm through diffusion equation,
	{\it Ann.\ Instit.\ Statist.\ Math.}\ {\bf 10}{\it (1)}
	(1959) 21--28.
%
\bibitem{Mueller} Mueller, Carl.
	On the support of solutions to the heat equation with noise, 
	\emph{Stochastics and Stochastics Rep.}\ 
	\textbf{37}{\it (4)} (1991) 225--245.
%
\bibitem{MuellerNualart} Mueller, Carl and David Nualart.
	Regularity of the density for the stochastic heat equation,
	{\it Electr.\ J. Probab.}\ {\bf 13}{\it (74)} (2008) 
	2248--2258. 
%
\bibitem{Naudts} Naudts, J.
	Dimension of discrete fractal spaces,
	{\it J. Phys.\ A} {\bf 21} (1988) 447--452.
%
\bibitem{OP} Orey, Steven, and William E. Pruitt.
	Sample functions of the $\rm N$-parameter Wiener process,
	\emph{Ann.\ Probab.}\ {\bf 1}{\it (1)} (1973) 138--163.
%
\bibitem{PaladinPelitiVulpiani} Paladin, G., L. Peliti, and A. Vulpiani.
	Intermittency as multifractality in history space,
	{\it J. Phys.\ A} {\bf 19} (1986) no.\ 16, L991--996.
%
\bibitem{PZ} Paley, R. E. A. C., and A. Zygmund.
	A note on analytic functions on the circle,
	{\it Proc.\ Cambridge Phil.\ Soc.}\ {\bf 28}{\it (3)} (1932) 266--272.
%
\bibitem{Pickands} Pickands, James, III.
	Upcrossing probabilities for stationary Gaussian processes,
	{\it Trans.\ Amer.\ Math.\ Soc.}\ {\bf 145} (1969) 51--73.
%
\bibitem{QW} Qualls, Clifford, and Hisao Watanbe.
	As asymptotic 0-1 behavior of Gaussian process,
	{\it Ann.\ Math.\ Statist.}\ {\bf 42}{\it (6)} (1971) 2029--2035.
%
\bibitem{Strassen} Strassen, V.
	An invariance principle for the law of the iterated logarithm,
	{\it Zeit. f\"ur Wahr. verw. Geb.} {\bf 3} (1964) 211--226.
%
\bibitem{Walsh} Walsh, John B.
	\emph{An Introduction to Stochastic Partial Differential Equations},
	in: \'Ecole d'\'et\'e de probabilit\'es de Saint-Flour, XIV---1984,
	265--439,
	Lecture Notes in Math., vol.\ 1180, Springer, Berlin, 1986.
%
\bibitem{Weber} Weber, M.
	Some examples of application of the metric entropy method,
	{\it Acta Math.\ Hungar.}\ {\bf 105}{\it (1-2)} (2004) 39--83.
%
\end{thebibliography}
\end{document}